\documentclass[12pt]{amsart}
\usepackage{fullpage,amssymb,hyperref,stmaryrd}
\numberwithin{equation}{subsection}

\newtheorem{theorem}[equation]{Theorem}
\newtheorem{lemma}[equation]{Lemma}
\newtheorem{cor}[equation]{Corollary}
\newtheorem{prop}[equation]{Proposition}
\newtheorem{conj}[equation]{Conjecture}
\theoremstyle{definition}
\newtheorem{hypothesis}[equation]{Hypothesis}
\newtheorem{defn}[equation]{Definition}
\newtheorem{remark}[equation]{Remark}
\newtheorem{example}[equation]{Example}

\newcommand{\be}{\mathbf{e}}
\newcommand{\bv}{\mathbf{v}}

\newcommand{\calC}{\mathcal{C}}
\newcommand{\calE}{\mathcal{E}}
\newcommand{\calF}{\mathcal{F}}
\newcommand{\calG}{\mathcal{G}}
\newcommand{\calH}{\mathcal{H}}
\newcommand{\calL}{\mathcal{L}}
\newcommand{\calO}{\mathcal{O}}
\newcommand{\calR}{\mathcal{R}}

\newcommand{\frako}{\mathfrak{o}}

\newcommand{\AAA}{\mathbb{A}}
\newcommand{\CC}{\mathbb{C}}

\newcommand{\GG}{\mathbb{G}}
\newcommand{\QQ}{\mathbb{Q}}
\newcommand{\RR}{\mathbb{R}}
\newcommand{\ZZ}{\mathbb{Z}}
\newcommand{\dual}{\vee}

\newcommand{\PP}{\mathbf{P}}

\DeclareMathOperator{\ana}{an}

\DeclareMathOperator{\FaFo}{FF}

\DeclareMathOperator{\Gal}{Gal}
\DeclareMathOperator{\GL}{GL}
\DeclareMathOperator{\inte}{int}
\DeclareMathOperator{\IR}{IR}
\DeclareMathOperator{\Lie}{Lie}

\DeclareMathOperator{\rank}{rank}
\DeclareMathOperator{\semis}{ss}
\DeclareMathOperator{\Spa}{Spa}
\DeclareMathOperator{\Spec}{Spec}

\begin{document}

\title{Monodromy representations of $p$-adic differential equations in families}
\author{Kiran S. Kedlaya}
\date{May 27, 2025}
\thanks{The author was supported by NSF (grant DMS-2053473) and UC San Diego (Warschawski Professorship). Thanks to Hansheng Diao, Koji Shimizu, Daxin Xu, Zijian Yao, and Gergely Z\'abr\'adi for helpful feedback.}

\begin{abstract}
We derive a relative version of the local monodromy theorem for ordinary differential equations on an annulus over a mixed-characteristic nonarchimedean field, and give several applications in $p$-adic cohomology and $p$-adic Hodge theory. These include a simplified proof of the semistable reduction theorem for overconvergent $F$-isocrystals, a relative version of Berger's theorem that de Rham representations are potentially semistable, and a multivariate version of the local monodromy theorem in the style of Drinfeld's lemma on fundamental groups.
\end{abstract}

\maketitle

Let $K$ be a nonarchimedean field (i.e., a field which is complete with respect to a specified nonarchimedean absolute value) of mixed characteristics $(0,p)$.
As presented in \cite{kedlaya-book}, there is a long-standing theory of \emph{ordinary differential equations} over $K$, by which we mean connections on vector bundles over a disc, annulus, or other smooth one-dimensional analytic space over $K$. A particularly important sector of this theory is the study of connections on an annulus which are ``solvable at a boundary''; such connections appear naturally in the study of coefficients in $p$-adic cohomology (namely \emph{overconvergent isocrystals}) and in $p$-adic Hodge theory (see below).
A central result is the \emph{$p$-adic local monodromy theorem}
\cite{andre-mono, kedlaya-locmono, mebkhout-mono}, which classifies solvable connections in terms of certain finite \'etale covers of the annulus; roughly speaking, these covers have ``wildly ramified reduction mod $p$'' which eliminates the ``irregularity'' of the connections in question.

The purpose of this paper is to gather some extensions of this theory to \emph{relative connections}, defined on the product of annulus with some base space, which are again (fiberwise) solvable at a boundary of the annulus. The main result is a relativization of the $p$-adic local monodromy theorem based on the form of the latter exposed in \cite[Chapter~22]{kedlaya-book}.
Compared to earlier incarnations of the $p$-adic local monodromy theorem (e.g., those cited above, or the presentation in the first edition of \cite{kedlaya-book}), the aforementioned version achieves two crucial improvements: it does not require $K$ to be discretely valued, and it does not require the connection to admit a Frobenius structure. This makes it feasible to 
carry out a relatively formal process to relativize the result using the quasicompactness of affinoid spaces. We anticipate a variety of applications of this result; here we limit ourselves to three general directions. (There is some potential to ``mix and match'' these directions, but we also eagerly expect some further applications in unforeseen directions.)

Our first application is a new and significantly simpler proof of the \emph{semistable reduction theorem} for overconvergent $F$-isocrystals \cite{kedlaya-semi4} (see Theorem~\ref{T:global semistable}, Remark~\ref{R:recover semistable reduction}). This result plays a foundational role in the study of \emph{rigid cohomology}, a Weil cohomology with $p$-adic coefficients introduced by Berthelot \cite{berthelot-mem}; in particular, the semistable reduction theorem is essential for the construction of a six functors formalism in which overconvergent $F$-isocrystals occur as the lisse coefficient objects \cite{abe-companion}.
The approach we take here is much more flexible than the original proof, and should therefore be more readily adaptable to variations of the semistable reduction theorem which may be needed in the future.

Our second application is a relative version of Berger's proof \cite{berger-mono} of Fontaine's conjecture $C_{\mathrm{st}}$: every $p$-adic Galois representation which is de Rham is also potentially semistable.
A relative version of this conjecture, for a de Rham $\ZZ_p$-local system on a smooth rigid analytic space over a finite extension of $\QQ_p$, has been proposed by Liu--Zhu \cite[Remark~1.4]{liu-zhu}. This was previously known locally in the neighborhood of a classical point
\cite{shimizu}; here, we obtain a similar statement locally around an arbitrary point (Theorem~\ref{T:relative de Rham}).

Our third application is a multivariate analogue of the $p$-adic local monodromy theorem for integrable connections on a product of annuli equipped with a full set of \emph{partial} Frobenius structures (Theorem~\ref{T:Drinfeld lemma form of local monodromy}). This can be thought of as a combination of the original local monodromy theorem with a local analogue of Drinfeld's lemma for $F$-isocrystals
\cite{kedlaya-dl-isocrystals, kedlaya-xu}. 
This statement can be further relativized (Theorem~\ref{T:Drinfeld lemma form of local monodromy relative}).

\section{Connections on annuli}

In this section, we collect and reformulate a number of results 
from \cite{kedlaya-book} concerning ordinary $p$-adic differential equations on an annulus.
We assume the reader either has a passing familiarity with \cite[Chapters~9--13]{kedlaya-book}
or has a copy of the text handy for easy reference.

We note in passing that the presentation in the aforementioned chapters of \cite{kedlaya-book} is much more algebraic in nature than the geometric perspective we will be adopting in subsequent sections.
For a discussion that bridges this gap, see \cite[Part~VII]{kedlaya-book}.

\subsection{Preliminaries}

\begin{defn}
Throughout this paper, let $K$ be a nonarchimedean field (i.e., a field complete with respect to a multiplicative absolute value, or equivalently a height-1 valuation) of mixed characteristics $(0,p)$.
Let $\kappa_K$ be the residue field of $K$. We use throughout Huber's model of analytic geometry in terms of \emph{adic spaces} over $K$ \cite{huber-book}, although much of our intuition will be derived from Berkovich's model \cite{berkovich1, berkovich2};
in particular, we will frequently and without comment use the fact that the topology on a complete Huber ring over $K$ can always be defined by some submultiplicative norm (see \cite[\S~1.5]{kedlaya-aws}). It would also be possible to use \emph{reified adic spaces} in the sense of \cite{kedlaya-reified}, but we ignore this point hereafter.

One consequence of our use of Huber's setting is that when we refer to an \emph{affinoid} adic space, we mean the space associated to a Huber pair over $K$ in which the underlying Huber ring need \emph{not} be a classical affinoid algebra in the sense of Tate; to obtain the latter, we add the condition \emph{tft} (topologically of finite type).

We further assume that the adic spaces we consider are always locally associated to Huber rings which are not merely \emph{sheafy} (i.e., the structure presheaf is a sheaf of rings)
but also \emph{strongly sheafy} in the sense of \cite{sousperfectoid}

For $X$ an adic space over $K$,
let $\overline{X}$ be the maximal Hausdorff quotient of $X$; open subsets of $\overline{X}$ correspond to \emph{partially proper} open subsets of $X$.
The projection map $X \to \overline{X}$ admits a canonical but discontinuous section $\overline{X} \to X$; the composition $X \to \overline{X} \to X$ takes each point to its unique height-1 generization.
For $x \in X$, let $\calH(x)$ denote the completed residue field of $x$; as the completion
is defined with respect to the topology defined by the valuation corresponding to $x$, which only depends on the underlying height-1 valuation, $\calH(x)$ depends only on the image of $X$ in $\overline{X}$. To recover $x$ from $\calH(x)$, we must also keep track of the induced valuation ring
$\calH(x)^+$ in $\calH(x)$; alternatively, $\calH(x)$ itself has a residue field $\kappa_{\calH(x)}$,
and $\calH(x)^+$ is the preimage of a certain valuation ring $\kappa_{\calH(x)}^+$ in $\kappa_{\calH(x)}$.

For $I$ an open or closed subinterval of $(0, \infty)$, write $A_K[I]$ for the annulus $|t| \in I$ over $K$. 
For $m$ a positive integer, let $A_K[I]^m$ be the $m$-fold product of $A_K[I]$ over $K$.
We omit the outer brackets when the interval is written out explicitly; for instance, if $I = [\epsilon,1)$, we expand the notation $A_K[I]$ as $A_K[\epsilon, 1)$ rather than $A_K[[\epsilon,1)]$.

For $\rho \in I$, let $\eta_\rho \in A_K[I]$ denote the \emph{Gauss point} of radius $\rho$,
i.e., the supremum norm over the circle $|t| = \rho$.
\end{defn}

\begin{defn}
Let $\calE$ be a \emph{connection  of rank $n$} on $A_K[I]$,
i.e., a vector bundle of rank $n$ equipped with an integrable connection.
(Since $A_K[I]$ is one-dimensional, the integrability condition is vacuous here.)
For each $\rho \in I$, let $\IR(\calE, \rho) \in (0,1]$ denote the \emph{intrinsic generic radius of convergence} of $\calE$ at $\rho$; here we use \cite[Definition~9.4.7]{kedlaya-book}
as the underlying definition, but thanks to  \cite[Theorem~11.9.2]{kedlaya-book} we may also interpret $\rho \IR(\calE, \rho)$ as the radius of the largest disc of $A_{\calH(\eta_\rho)}[I]$ centered at the canonical rational point over $\eta_\rho$ on which $\calE$ admits a basis of horizontal sections.
In particular, $\IR(\calE_1 \otimes \calE_2, \rho) \leq \max\{\IR(\calE_1, \rho), \IR(\calE_2, \rho)\}$ and $\IR(\calE, \rho) = \IR(\calE^\dual, \rho)$ (compare  \cite[Lemma~9.4.6]{kedlaya-book}); moreover, the definition of $\IR(\calE, \rho)$ is invariant under base extension on $K$ \cite[Remark~9.4.9]{kedlaya-book}.

For $n = \rank \calE$, we also define the \emph{intrinsic subsidiary radii of convergence}
$0 < s_1(\calE, \rho) \leq \cdots \leq s_n(\calE, \rho) \leq 1$ of $\calE$ at $\rho$ by the property that $\rho s_i(\calE, \rho)$ is the radius of the largest disc centered at $\eta_\rho$ on which $\calE$ admits $n-i+1$ linearly independent horizontal sections
\cite[Definition~9.8.1]{kedlaya-book}. In particular,  $s_1(\calE, \rho) = \IR(\calE, \rho)$.
\end{defn}

\begin{theorem} \label{T:convexity}
For $i=1,\dots,n$, the function $r \mapsto -\log s_1(\calE, e^{-r})- \cdots - \log s_i(\calE, e^{-r})$ is continuous, convex, and piecewise affine with slopes in $\frac{1}{n!} \ZZ$.
\end{theorem}
\begin{proof}
See \cite[Theorem~11.3.2]{kedlaya-book}.
\end{proof}

\begin{lemma} \label{L:unipotent to log}
For $I$ an open interval, $\calE$ is unipotent (i.e., a successive extension of trivial connections) if and only if $M_{\log} := \Gamma(A_K[I], \calE) \otimes_{\calO(A_K[I])} \calO(A_K[I])[\log t]$ admits a horizontal basis.
\end{lemma}
\begin{proof}
For the ``only if'' direction, it suffices to identify the Yoneda extension group of two trivial rank-1 differential modules over $\calO(A_K[I])[\log t]$ with the cokernel of $\frac{d}{dt}$, and then to observe that the latter is trivial: the monomial $t^i (\log t)^j$ antidifferentiates to $\frac{1}{j+1} (\log t)^{j+1}$ if $i  = -1$, and otherwise to 
$\frac{1}{i+1} t^{i+1} (\log t)^j$ minus the antiderivative of $\frac{j}{i+1} t^i (\log t)^{j-1}$.

For the ``if'' direction, 
it is sufficient to verify that if $\calE \neq 0$ and $M_{\log}$ admits a horizontal basis, then $H^0(\calE) \neq 0$. For this, pick any nonzero horizontal section of $M_{\log}$ and let $i$ be the largest power of $\log t$ that occurs in this section. If we write this section as a polynomial in $\log t$ with coefficients in $\Gamma(A_K[I], \calE)$, then the coefficient of $(\log t)^i$ is killed by the action of $\frac{d}{dt} + i$; consequently, multiplying this coefficient by $t^{-i}$ yields a nonzero horizontal section of $\calE$.
(Compare \cite[Exercise~18.2]{kedlaya-book}.)
\end{proof}

\subsection{Regular connections}
\label{subsec:regular conn}

\begin{defn}
For $I$ an open interval, we say that a connection $\calE$
on $A_K[I]$ is \emph{regular} if $\IR(\calE, \rho) = 1$ for all $\rho \in I$.
This condition implies the existence of an action of $K^\times$ on $\calE$ via the substitutions $t \mapsto \lambda t$ (compare the discussion in \cite[Definition~13.5.2]{kedlaya-book}).

In the literature on $p$-adic differential equations, regular objects are also said to satify the \emph{Robba condition} \cite[Chapter~13]{kedlaya-book}.
Our terminology here is meant to convey an analogy between these objects and regular connections in the theory of ordinary differential equations over $\CC$ \cite[Chapter~7]{kedlaya-book}.
\end{defn}

\begin{defn}
We say that two tuples $A,B \in \ZZ_p^n$ are \emph{equivalent} if they belong to the same orbit under the wreath product $\ZZ \wr S_n$ (i.e., the group generated by diagonal integer shifts and permutation of indices).

For $x \in \QQ_p$, let $\langle x \rangle$ denote the smallest (for the usual absolute value) element of $\ZZ[p^{-1}]$ congruent to $x$ modulo $\ZZ_p$. (This is ambiguous when $p=2$ and $x \equiv \frac{1}{2} \pmod{\ZZ_p}$, in which case either choice is fine for what follows.)
Following \cite[Definition~13.4.2]{kedlaya-book},
we say that $A$ and $B$ are \emph{weakly equivalent} if there exist a constant $c > 0$
and a sequence $\sigma_1, \sigma_2, \dots$ of permutations of $\{1,\dots,n\}$ such that
\[
\left| p^m \left\langle \frac{A_i - B_{\sigma_m(i)}}{p^m} \right\rangle \right| \leq cm \qquad (i=1,\dots,n; m=1,2,\dots) 
\]
This holds if $A$ and $B$ are equivalent, but not conversely \cite[Example~13.4.6]{kedlaya-book}. See however Lemma~\ref{L:weakly equivalent to equivalent}.
\end{defn}

\begin{lemma} \label{L:weakly equivalent to equivalent}
Let $A \in \ZZ_p^n$ be a tuple with \emph{$p$-adic non-Liouville differences}, that is, the difference between any two elements of $A$ is not a $p$-adic Liouville number 
\cite[Definition~13.1.2]{kedlaya-book}. Then any $B \in \ZZ_p^n$ which is weakly equivalent to $A$ is also equivalent to $A$.
\end{lemma}
\begin{proof}
See \cite[Proposition~13.4.5]{kedlaya-book}.
\end{proof}

\begin{defn}
Let $I$ be an open interval, let $\calE$ be a regular connection of rank $n$ on $A_K[I]$,
and fix a closed subinterval $J$ of $I$.
We define an \emph{exponent} for $\calE$ on $J$ to be a tuple $A \in \ZZ_p^n$ as in \cite[Definition~13.5.2]{kedlaya-book}. 
If $A$ can be chosen uniformly over $J$, we say that it is an \emph{exponent for $\calE$}.
\end{defn}

\begin{remark}
In lieu of including more details of the definition of exponents, we recall some key properties.
\begin{itemize}
\item
If $A, B \in \ZZ_p^n$ are equivalent and $A$ is an exponent for $\calE$ on $J$, then so is $B$.
\item
If $A$ is an exponent for $\calE$ on $J$, then $A$ is also an exponent for $\calE$ on any closed subinterval of $J$.
\item
If $\calE$ is trivial, then the zero tuple is an exponent for $\calE$.
\item
If $0 \to \calE_1 \to \calE \to \calE_2 \to 0$ is an exact sequence of regular connections
and $A_i$ is an exponent for $\calE_i$ on $J$, then the concatenation of $A_1$ and $A_2$ is an exponent for $\calE$ on $J$. In particular, if $\calE$ is unipotent, then the zero tuple is an exponent for $\calE$. The converse also holds but is much deeper; see Theorem~\ref{T:CM}.
\end{itemize}
\end{remark}

\begin{theorem}[Christol--Mebkhout] \label{T:CM}
Let $I$ be an open interval, let $\calE$ be a regular connection of rank $n$ on $A_K[I]$,
and fix a closed subinterval $J$ of $I$.
\begin{enumerate}
\item[(a)]
There exists an exponent for $\calE$ on $J$. 
\item[(b)]
If $J$ is of positive length, then any two exponents for $\calE$ on $J$ are weakly equivalent
in the sense of \cite[Definition~13.4.2]{kedlaya-book}.
In particular, if there is an exponent with $p$-adic non-Liouville differences, then any other exponent is equivalent to it.
\item[(c)]
If the zero $n$-tuple is an exponent for $\calE$, then $\calE$ is unipotent.
\end{enumerate}
\end{theorem}
\begin{proof}
See \cite[Theorem~13.5.5, Theorem~13.5.6, Theorem~13.6.1]{kedlaya-book}.
\end{proof}

\begin{cor} \label{C:extend exponent}
Let $I$ be an open interval and let $\calE$ be a regular connection of rank $n$ on $A_K[I]$.
Let $A \in \ZZ_p^n$ be a tuple with $p$-adic non-Liouville differences. If there exists some closed subinterval $J$ of $I$ of positive length such that $A$ is an exponent for $\calE$ on $J$, then $A$ is an exponent for $\calE$.
\end{cor}
\begin{proof}
It suffices to check that $A$ is an exponent for $\calE$ on any closed subinterval $J'$ containing $J$, as we may then restrict to any closed subinterval of $J'$ and thereby cover all closed subintervals of $I$.
By Theorem~\ref{T:CM}(a), there exists an exponent $B \in \ZZ_p^n$ for $\calE$ on $J'$, which is then also an exponent for $\calE$ on $J$. By Theorem~\ref{T:CM}(b), $A$ and $B$ are weakly equivalent, and hence equivalent by our condition on (a) plus Lemma~\ref{L:weakly equivalent to equivalent}. Hence $A$ is also an exponent for $\calE$ on $J'$, as claimed.
\end{proof}

\subsection{Solvable connections}

\begin{defn}
For $I = [\epsilon,1)$ or $I = (\epsilon, 1)$, we say that a connection $\calE$
on $A_K[I]$ is \emph{solvable (at $1$)}
if $\IR(\calE, \rho) \to 1$ as $\rho \to 1^-$.
For $\calE$ solvable of rank $n$,
for $i=1,\dots,n$, let $b_i(\calE)$ be the slope of the function $r \mapsto -\log s_i(\calE, e^{-r})$
as $r \to 0^+$; we sometimes write $b(\calE)$ instead of $b_1(\calE)$.
By Theorem~\ref{T:convexity}, $b_i(\calE)$ is well-defined and belongs to $\frac{1}{n!} \ZZ_{\geq 0}$.
\end{defn}

\begin{defn}
Let $\calC_K$ denote the 2-colimit of the categories of solvable connections on $A_K(\epsilon,1)$
(as full subcategories of the categories of arbitrary connections)
over all $\epsilon \in (0,1)$, or equivalently on $A_K[\epsilon,1)$ over all $\epsilon \in (0,1)$.

For $\calE \in \calC_K$ of rank $n$, the intrinsic subsidiary radii of convergence of $\calE$ at $\rho$ are not well-defined as functions on any specific interval $(\epsilon, 1)$, but they are well-defined 
as germs at $1^-$. In particular, 
the quantities $b_1(\calE), \dots, b_n(\calE)$ are well-defined and belong to $\frac{1}{n!}\ZZ_{\geq 0}$.
\end{defn}

\begin{lemma} \label{L:criterion for big slope}
Let $\calE$ be a solvable connection of rank $n$ on $A_K[\epsilon,1)$ for some $\epsilon > 0$.
For $c \in \frac{1}{n!} \ZZ$, $i \in \{1,\dots,n\}$,
$b_1(\calE) + \cdots + b_i(\calE) > c$ if and only if
\begin{equation} \label{eq:criterion for big slope}
s_1(\calE, \rho) \cdots s_i(\calE, \rho) \leq \rho^{c + 1/n!} \qquad (\rho \in [\epsilon, 1)).
\end{equation}
\end{lemma}
\begin{proof}
For the ``if'' direction, note that $s_1(\calE, \rho) \cdots s_i(\calE, \rho) = \rho^{b_1(\calE) + \cdots + b_i(\calE)}$ for $\rho$ sufficiently close to 1.
For the ``only if'' direction, suppose that $b_1(\calE_x) + \cdots + b_i(\calE_x) > c$; then
the function $r \mapsto -\log s_1(\calE_x, e^{-r}) - \cdots - \log s_i(\calE_x, e^{-r})$ 
on $(0, -\log \epsilon]$ is convex, tends to 0 as $r \to 0^+$,
and by Theorem~\ref{T:convexity} is convex and piecewise affine with slope at 0 at least
$c + \frac{1}{n!}$. It follows that the function is bounded below by $(c + \frac{1}{n!})r$ for all $r$, proving the claim.
\end{proof}

\begin{theorem}[Christol--Mebkhout] \label{T:christol-mebkhout}
For $\calE \in \calC_K$, there exists a unique direct sum decomposition 
(called the \emph{slope decomposition})
\[
\calE = \bigoplus_{s \in \QQ_{\geq 0}} \calE_s
\]
such that for $s \in \QQ_{\geq 0}$, $i \in \{1,\dots,\rank(\calE_s)\}$ we have
$b_i(\calE_s) = s$.
\end{theorem}
\begin{proof}
See \cite[Theorem~12.6.4]{kedlaya-book}.
\end{proof}

\begin{defn}
For $\calE \in \calC_K$ of rank $n$, we define the \emph{slopes} of $\calE$ as the multisubset of $\QQ_{\geq 0}$ of cardinality $n$ containing $s$ with multiplicity equal to $\rank \calE_s$.
We refer to $\calE_s$ as the \emph{slope-$s$ component} of $\calE$; in case $\calE = \calE_s \neq 0$, we say that $\calE$ is a \emph{slope-$s$ object} of $\calC_K$.
We also refer to slope-$0$ components and objects as \emph{regular} components and objects.
\end{defn}

\subsection{Quasiconstant connections}

\begin{defn}
For $\alpha \in I$, let $\left|\bullet\right|_\alpha$ denote the Gauss norm on the ring $\calO(A_K[I])$.
The elements of $\calO(A_K[I])$ correspond to power series $x = \sum_{i \in \ZZ} x_i t^i$ with $x_i \in K$ subject to the appropriate convergence condition, and we have $|x|_\alpha = \sup_i \{|x_i| \alpha^i\}$.

Let $\calR_K$ be the colimit of the rings $\calO(A_K[\epsilon,1))$ over all $\epsilon \in (0,1)$.
Let $\calR_K^{\inte}$ be the subring of $\calR_K$ consisting of series $x = \sum_{i \in \ZZ} x_i t^i$ for which $|x_i| \leq 1$ for all $i$; equivalently, these are the series $x$ for which
$\limsup_{\alpha \to 1^-} |x|_\alpha \leq 1$.
\end{defn}

\begin{remark}
The ring $\calR_K^{\inte}$ is local with residue field $\kappa_K((\overline{t}))$.
Every $x \in \calR_K^{\inte}$ with $|x|_1 < 1$ belongs to the maximal ideal,
but not conversely unless $K$ is discretely valued.
\end{remark}

\begin{lemma} \label{L:reduce mod to norm}
For $x \in \calR_K^{\inte}$, the following statements hold.
\begin{enumerate}
\item[(a)]
If $x$ belongs to the preimage of $\kappa_K \llbracket \overline{t} \rrbracket$, then there exists
$\alpha_0 \in (0,1)$ such that $|x|_\alpha \leq 1$ for all $\alpha \in (\alpha_0, 1)$.
\item[(b)]
If $x$ belongs to the maximal ideal of $\calR_K^{\inte}$, there exists
$\alpha_0 \in (0,1)$ such that $|x|_\alpha < 1$ for all $\alpha \in (\alpha_0, 1)$.
\end{enumerate}
\end{lemma}
\begin{proof}
Write $x = \sum_{i \in \ZZ} x_{i} t^i$. 
We first treat (b).
By hypothesis, we have
$|x_{i}| < 1$ for all $i$; consequently, for every $\alpha \in (0,1)$ we have $|x_{i}| \alpha^i < 1$ for all $i \geq 0$. On the other hand, for some $\alpha_1 \in (0,1)$, the quantities
$|x_{i}| \alpha_1^i$ are uniformly bounded over all $i < 0$. By replacing $\alpha_1$ with any value $\alpha_2 \in (\alpha_1,1)$, we achieve that $|x_{i}| \alpha_2^i \to 0$ as $i \to -\infty$, and so
the set $S$ of $i<0$ for which for which $|x_{i}| \alpha_2^i \geq 1$ is finite. We may thus choose $\alpha_0 \in (\alpha_2,1)$ so that $|x_{1}| \alpha_0^i < 1$ for each $i \in S$, and this does the job.

To treat (a), put $y = \sum_{i \geq 0} x_i t^i$; then $x-y$ belongs to the maximal ideal of $\calR_K^{\inte}$, so by (b) there exists $\alpha_0 \in (0,1)$ such that $|x-y|_\alpha < 1$
for all $\alpha \in (\alpha_0, 1)$. On the other hand, since $|x_i| \leq 1$ for all $i$, it is evident that $|y|_\alpha \leq 1$ and hence $|x|_\alpha \leq 1$.
\end{proof}

\begin{lemma} \label{L:henselian}
The local ring $\calR_K^{\inte}$ is henselian.
Consequently, base extension induces an equivalence of categories between finite \'etale algebras
over $\calR_K^{\inte}$ and over $\kappa_K((\overline{t}))$.
\end{lemma}
\begin{proof}
This is asserted in \cite[Lemma~22.1.2(b)]{kedlaya-book}, but the proof is only explained in detail when $K$ is discretely valued (see \cite[Lemma~15.1.3(c)]{kedlaya-book}). 
We give a more detailed explanation here.

It suffices to check that for any monic polynomial $P(x) \in \calR_K^{\inte}[x]$ and any simple root $\overline{r} \in \kappa_K((\overline{t}))$ of $\overline{P}(x) \in \kappa_K((\overline{t}))[x]$, there exists a  (necessarily unique) root $r \in \calR_K^{\inte}$ of $P(x)$ lifting $\overline{r}$.
By a variable substitution on $P(x)$, we may reduce to the case where
$\overline{r} = 0$, $\overline{P}(x) \in \kappa_K \llbracket \overline{t} \rrbracket[x]$,
and $\overline{P}'(\overline{r}) = t^m$ for some $m \in \ZZ$.

Write $P(x) = \sum_{i=0}^n c_i x^i$ with $c_n = 1$, $|c_1-t^m| < 1$, and $|c_0| < 1$.
For $\alpha \in (0,1)$ sufficiently large, the coefficients of $P(x)$ belong to $\calO(A_K[\alpha,1))$; moreover, by Lemma~\ref{L:reduce mod to norm}, for suitable $\alpha$ we also have 
$|c_i|_\alpha \leq 1$ for all $i$, $|c_0/t^{2m}|_\alpha < 1$, and $|t^m - c_1|_\alpha < 1$.
In particular we have $|P(0)|_\alpha < |P'(0)|_\alpha^2$, so we may compute the Newton--Raphson iteration starting from $0$ in the ring of $y \in \calO(A_K[\alpha, 1))$ for which $|y|_\beta \leq 1$ for all $ \beta \in [\alpha,1)$; the limit is the desired root $r$.
\end{proof}

\begin{defn}
For $\epsilon \in (0,1)$, a finite \'etale covering $Y \to A_K(\epsilon,1)$ is \emph{eligible} if it is induced by a finite \'etale algebra over $\calR^{\inte}_K$; this algebra can then be identified with
\[
\bigcup_{\epsilon' \in [\epsilon,1)} \{f \in \calO(Y \times_{A_K(\epsilon,1)} A_K(\epsilon',1))\colon
\limsup_{\alpha \to 1^-} |f|_\alpha \leq 1\}.
\]
Suppose that $\kappa_K$ is perfect. Then every finite separable extension $L$ of $\kappa_K((\overline{t}))$ can be written as $\kappa'((\overline{u}))$ for some  finite (separable) extension $\kappa'$ of $\kappa_K$ and some $\overline{u}$. 
Consequently, the corresponding connected eligible covering of $A_K(\epsilon,1)$ can (after increasing $\epsilon$) be itself viewed as an annulus over the unramified extension of $K$ with residue field $\kappa'$. We refer to a pullback along such a covering as a \emph{pullback along $L$}.

If $\kappa_K$ is not perfect, then we may apply similar logic if we allow ourselves to first make a base change from $K$ to a suitable finite extension $K'$ (depending on $L$), in order to enforce that the residue field of any component of $L \otimes_{\kappa_K((t))} \kappa_{K'}((t))$ is separable over $\kappa_{K'}$.
\end{defn}

\begin{defn}
For $\calE \in \calC_K$, we say that $\calE$ is \emph{quasiconstant} if there exists some finite separable extension $L$ of $\kappa_K((\overline{t}))$ such that,
after replacing $K$ with some extension field $K'$ for which the pullback of $\calE$ along $L$ is defined, this pullback is a trivial connection. Since the triviality of a connection can be checked after a base extension, the choice of $K'$ does not matter.
\end{defn}

\begin{defn} \label{D:wild mono1}
Suppose that $\calE \in \calC_K$ is quasiconstant.
If $\kappa_K$ is perfect, then we may choose a finite Galois extension $L$ of $\kappa_K((\overline{t}))$ such that the pullback of $\calE$ along $L$
is a trivial connection.
The Galois group $G$ of this extension acts on the space of horizontal sections of the pullback of $\calE$ along $L$.
For $I$ the inertia subgroup of $G$, we obtain a representation $\tilde{\rho}\colon I \to \GL_n(K')$ for $n = \rank \calE$, where $K'$ is some finite (unramified) extension of $K$. This then restricts to a representation $\rho\colon I_{\kappa_K((\overline{t}))} \to \GL_n(\overline{K})$
for $I_{\kappa_K((\overline{t}))}$ the inertia subgroup of $G_{\kappa_K((\overline{t}))}$;
we call this the \emph{monodromy representation} associated to $\calE$.
Let $W_{\kappa_K((\overline{t}))}$ denote the wild inertia subgroup of $I_{\kappa_K((\overline{t}))}$; we call the restriction of $\rho$ to $W_{\kappa_K((\overline{t}))}$ the \emph{wild monodromy representation} associated to $\calE$.

The definition of the monodromy representation and the wild monodromy representation commute with extension of the base field. We may thus extend them via base extension to the case where
$\kappa_K$ is not perfect.
\end{defn}

\begin{theorem}[Matsuda] \label{T:matsuda}
Suppose that $\kappa_K$ is perfect.
For $\calE \in \calC_K$ quasiconstant, let $\rho$ be the wild monodromy representation of $\calE$.
Then for $s \in \QQ$, the rank of the slope-$s$ component of $\calE$ equals
the multiplicity of $s$ as a ramification break of $\rho$.
In particular, $\calE$ is regular if and only if the pullback of $\calE$ along some tame extension of $\kappa_K((\overline{t}))$ is trivial.
\end{theorem}
\begin{proof}
Note that $\calE$ can always be realized using a discretely valued coefficient field.
Consequently, we may apply \cite[Theorem~19.4.1]{kedlaya-book} and references therein.
\end{proof}

\subsection{The \texorpdfstring{$p$-adic}{p-adic} local monodromy theorem}
\label{sec:p-adic mono}

We are now ready to formulate the $p$-adic local monodromy theorem. 
Throughout \S\ref{sec:p-adic mono}, assume that $\kappa_K$ is perfect.
\begin{theorem} \label{T:p-adic mono}
For $\calE \in \calC_K$,
there exists a finite separable extension $L$ of $\kappa_K((\overline{t}))$
(depending on $\calE$) such that the pullback of $\calE$ along $L$ is regular.
\end{theorem}
\begin{proof}
See \cite[Theorem~22.1.4]{kedlaya-book}.
Note that crucially, the formulation of this theorem in the second edition of \cite{kedlaya-book} does not require $K$ to be discretely valued, in contrast with older versions of the theorem (such as in the first edition of \cite{kedlaya-book}).
\end{proof}

\begin{cor} \label{C:quasiconstant from cover}
For $\calE \in \calC_K$, suppose that there exists a (not necessarily eligible) finite \'etale cover $Y$ of $A_K(\epsilon,1)$ such that the pullback of $\calE$ to $Y$ is trivial. Then $\calE$ is quasiconstant.
\end{cor}
\begin{proof}
Using Theorem~\ref{T:p-adic mono}, we may reduce to the case where $\calE$ is regular.
We may further assume that $\pi\colon Y \to A_K(\epsilon,1)$ is Galois and that $\pi_* \calO_Y$, as a connection on $A_K(\epsilon,1)$, is regular. 
We may now argue as in \cite[Lemma~6.12]{kedlaya-simpleconn} to deduce that $\calE$
becomes constant under pullback along a tame extension of $\kappa_K((\overline{t}))$.
\end{proof}

\begin{cor} \label{C:eligible by pushforward}
Let $\pi\colon Y \to A_K(\epsilon,1)$ be a finite \'etale cover such that $\pi_* \calO_Y$, as a connection on $A_K(\epsilon,1)$, is solvable. Then $\pi$ is an eligible cover.
\end{cor}
\begin{proof}
By Corollary~\ref{C:quasiconstant from cover}, we may reduce to the case where $\pi_* \calO_Y$ is in fact constant. In this case, we may use the horizontal sections of $\pi_* \calO_Y$ to show that $\pi$ splits completely.
\end{proof}

As an application, we can recover the original formulation of the $p$-adic local monodromy theorem \cite[Theorem~20.1.4]{kedlaya-book}, but without the restriction that $K$ be discretely valued.
\begin{defn}
For $q$ a power of $p$,
a \emph{relative ($q$-power) Frobenius lift} on $A_K(\epsilon,1)$
is the composition of an endomorphism of $A_K(\epsilon,1)$ induced by an isometric endomorphism of $K$ with a $K$-linear map
$\varphi\colon A_K(\epsilon,1) \to A_K(\epsilon^{1/q}, 1)$ with the property that for some power $q$ of $p$, $|\varphi^*(t) - t^q|_\rho < 1$ for $\rho \in (0,1)$ sufficiently close to 1.

A \emph{relative ($q$-power) Frobenius structure} on an object $\calE \in \calC_K$ is an isomorphism $\varphi^* \calE \cong \calE$ for some relative ($q$-power) Frobenius lift $\varphi$ on $A_K(\epsilon,1)$.
Using the Taylor isomorphism as in \cite[Proposition~17.3.1]{kedlaya-book},
we may canonically transform relative Frobenius structures between different choices of $\varphi$.
\end{defn}

\begin{lemma} \label{L:regular Frobenius equals tame}
Any regular object of $\calC_K$ admitting a relative Frobenius structure becomes unipotent (i.e., a successive extension of trivial connections) after pullback along a tame extension of $\kappa_K((\overline{t}))$.
\end{lemma}
\begin{proof}
We may assume that the underlying relative Frobenius lift $\varphi$ has the form $\varphi(t) = t^q$. In this case, this becomes an application of the theory of $p$-adic exponents; see \cite[Proposition~13.6.2]{kedlaya-book}.
\end{proof}
\begin{cor} \label{C:unipotent pure slope} 
Suppose that $K$ is discretely valued.
Any regular object of $\calC_K$ admitting a relative Frobenius structure \emph{of pure slope}
becomes trivial (not just unipotent) after pullback along a tame extension of $\kappa_K((\overline{t}))$.
\end{cor}
\begin{proof}
By Lemma~\ref{L:regular Frobenius equals tame}, we may start with a unipotent object $\calE$.
By \cite[Lemma~9.2.3]{kedlaya-book}, $\calE$ admits a basis $\be_1,\dots,\be_n$ on which $t \frac{d}{dt}$ acts via a nilpotent matrix $N$ over $K$.
Let $A$ be the matrix via which the Frobenius structure acts on the same basis; then
\[
NA + t\frac{d}{dt}(A) = p A \varphi_K^*(N).
\]
Write $A = \sum_i A_i t^i$, so that for each $i$ we have
\[
NA_i + i A_i = p A_i \varphi_K^* (N).
\]
The operator $A \mapsto NA - p A\varphi_K^*(N)$ is then nilpotent: applying it $2n-1$ times yields an expression in which every summand is either divisible by $N^n$ on the left or by $\varphi_K^*(N)^n$ on the right. We deduce that $A_i = 0$ for $i > 0$, and so 
\[
NA = p A \varphi_K^*(N).
\]
At this point we are free to enlarge $K$, so we may assume that its residue field is strongly difference-closed \cite[Remark~14.3.5]{kedlaya-book}. By the pure slope hypothesis, we may then rechoose the basis $\be_1,\dots,\be_n$ so that $A$ becomes a scalar matrix \cite[Theorem~14.6.3]{kedlaya-book}. The previous equation then becomes
$N = p \varphi_K^*(N)$; since $\varphi_K^*$ is an isometry, this forces $N = 0$
and thus proves the claim.
\end{proof}

\begin{theorem} \label{T:p-adic mono with Frob}
Every object of $\calC_K$ admitting a relative Frobenius structure becomes unipotent after pullback along some finite separable extension of $\kappa_K((\overline{t}))$.
\end{theorem}
\begin{proof}
Combine Theorem~\ref{T:p-adic mono} with Lemma~\ref{L:regular Frobenius equals tame}.
\end{proof}

\subsection{Wild monodromy representations}

We next move towards a reformulation of the local monodromy theorem in which we avoid explicitly referencing an eligible cover.

\begin{defn} \label{D:equivalence by slope}
For $\calE_1, \calE_2 \in \calC_K$, let $\calF$ be the regular component of $\calE_1^\dual \otimes \calE_2$. We say that $\calE_1 \sim \calE_2$ if the contractions
\begin{gather*}
\calF \otimes \calE_1 \to \calE_1^\dual \otimes \calE_2 \otimes \calE_1 
\cong (\calE_1^\dual \otimes \calE_1) \otimes \calE_2
\to \calE_2 \\
\calF^\dual \otimes \calE_2 \to \calE_2^\dual \otimes \calE_1 \otimes \calE_2 
\cong (\calE_2^\dual \otimes \calE_2) \otimes \calE_1
\to \calE_1
\end{gather*}
are surjective. 
\end{defn}

\begin{lemma}
In Definition~\ref{D:equivalence by slope}, the relation $\sim$ defines an equivalence relation
on isomorphism classes of $\calC_K$.
\end{lemma}
\begin{proof}
If $\calE_1 \cong \calE_2$, then $\calF$ contains a trivial connection $\calG$ generated by the isomorphism (more precisely, by the horizontal section of $\calE_1^\dual \otimes \calE_2$ corresponding to the isomorphism). The contraction $\calG \otimes \calE_1 \to \calE_2$ is then an isomorphism, so
$\calE_1 \sim \calE_2$. In particular, $\sim$ defines a reflexive and manifestly symmetric relation; it thus remains to check transitivity.

To this end, suppose that $\calE_1 \sim \calE_2$ and $\calE_2 \sim \calE_3$. Let $\calG$ be the slope-$0$ component of $\calE_2^\dual \otimes \calE_3$
and let $\calH$ be the slope-0 component of 
$\calE_1^\dual \otimes \calE_3$. Then the contraction 
\[
(\calE_1^\dual \otimes \calE_2) \otimes (\calE_2^\dual \otimes \calE_3) \to \calE_1^\dual \otimes \calE_3
\]
carries $\calF \otimes \calG$ to $\calH$. Hence the map
\[
\calF \otimes \calG \otimes \calE_1 \to \calG \otimes \calE_2 \to \calE_3,
\]
which is written as the composition of two surjections, factors through
$\calH \otimes \calE_1 \to \calE_3$ and thus forces this map to also be surjective.
By similar logic, $\calH^\dual \otimes \calE_3 \to \calE_1$ is also surjective.
\end{proof}

\begin{remark}
For any $\calE \in \calC_K$ and any positive integer $n$, $\calE \sim \calE^{\oplus n}$.
From this we see that $\sim$ does not preserve rank.
\end{remark}

\begin{remark} \label{R:Robba condition via equivalence}
If $\calE_1 \sim \calE_2 \in \calC_K$ and $\calE_1$ is regular, then 
$\calE_2$ is a quotient of $\calF \otimes \calE_1$ and both of the factors in the tensor product are regular. Consequently, $\calE_2$ is also regular.
\end{remark}

\begin{lemma} \label{L:semisimplification equivalence}
For $\calE \in \calC_K$ with semisimplification $\calE^{\semis}$, 
we have $\calE \sim \calE^{\semis}$.
\end{lemma}
\begin{proof}
It suffices to check that for any exact sequence
\[
0 \to \calE_1 \to \calE \to \calE_2 \to 0
\]
we have $\calE \sim \calE_1 \otimes \calE_2$. Tensoring with $(\calE_1 \oplus \calE_2)^\dual$ yields another exact sequence
\[
0 \to (\calE_1^\dual \otimes \calE_1) \oplus (\calE_2^\dual \otimes \calE_1) \to (\calE_1 \oplus \calE_2)^\dual \otimes \calE \to (\calE_1^\dual \otimes \calE_2) \oplus (\calE_2^\dual \otimes \calE_2) \to 0.
\]
Let $\calF$ be the regular component of $(\calE_1 \oplus \calE_2)^\dual \otimes \calE$;
then $\calF$
admits the trace component of $\calE_1^\dual \otimes \calE_1$ as a subobject
and
the trace component of $\calE_2^\dual \otimes \calE_2$ as a quotient.
This implies that the image of $\calF \otimes (\calE_1 \oplus \calE_2) \to \calE$ 
contains $\calE_1$ and that the composition  $\calF \otimes (\calE_1 \oplus \calE_2) \to \calE \to \calE_2$ is surjective, 
so  $\calF \otimes (\calE_1 \oplus \calE_2) \to \calE$  is surjective.

Similarly,
$\calF^\dual$
admits the trace component of $\calE_1^\dual \otimes \calE_1$ as a quotient
and
the trace component of $\calE_2^\dual \otimes \calE_2$ as a subobject.
This implies that the image of $\calF^\dual \otimes \calE \to \calE_1 \oplus \calE_2$
contains $\calE_2$ and that the composition $\calF^\dual \otimes \calE \to \calE_1 \oplus \calE_2 \to \calE_1$ is surjective, so $\calF^\dual \otimes \calE \to \calE_1 \oplus \calE_2$ is also surjective.
\end{proof}

\begin{prop} \label{P:quasiconstant rep}
For any $\calE_1 \in \calC_K$, there exists a quasiconstant object $\calE_2 \in \calC_K$ 
such that $\calE_1 \sim \calE_2$.
\end{prop}
\begin{proof}
Suppose first that $\kappa_K$ is perfect.
By Lemma~\ref{L:semisimplification equivalence}, we may assume that $\calE_1$ is irreducible.
Apply Theorem~\ref{T:p-adic mono} to choose a finite  separable extension $L$ of $\kappa_K((\overline{t}))$
such that the pullback of $\calE_1$ along $L$ is regular.
Let $\calE'$ be the pushforward of the trivial connection from $L$ to $\kappa_K((\overline{t}))$;
note that $\calE'$ is isomorphic to its own dual.
The regular component of $\calE_1^\dual \otimes \calE'$ is nonzero:
for a generic disc $D$ of radius sufficiently close to $1$, 
$\calE_1^\dual$ is trivial on the pullback of $D$, and any nonzero horizontal section
defines a horizontal section of $\calE_1^\dual \otimes \calE'$.
Since $\calE'$ is semisimple, there is an irreducible subobject $\calE_2$ of $\calE'$
such that the regular component $\calF$ of $\calE_1^\dual \otimes \calE_2$ is nonzero.
Since the maps $\calF \otimes \calE_1 \to \calE_2$, $\calF^\dual \otimes \calE_2 \to \calE_1$
are nonzero and their targets are irreducible, both maps are surjective.

Suppose now that $\kappa_K$ is general.
We may then apply the previous argument after replacing $K$ with a suitable finite extension $K'$. We may then taking the resulting object $\calE_2$, take its pushforward from $K'$ to $K$, then choose an irreducible constituent of the result to achieve the desired goal.
\end{proof}

\begin{lemma}
For $\calE_1, \calE_2 \in \calC_K$ quasiconstant and irreducible, if $\calE_1 \sim \calE_2$, then the wild monodromy representations associated to $\calE_1$ and $\calE_2$ are isomorphic.
\end{lemma}
\begin{proof}
We may assume that $\kappa_K$ is perfect.
Let $\rho_1$ and $\rho_2$ be the wild monodromy representations associated to $\calE_1$ and $\calE_2$. By hypothesis,
$\calE_1^\dual \otimes \calE_2$ contains a nonzero regular component;
by Theorem~\ref{T:matsuda}, this implies that $\rho_1^\dual \otimes \rho_2$ contains a trivial subrepresentation. This implies the desired isomorphism.
\end{proof}

\begin{lemma} \label{L:rank}
Suppose that $\calE \in \calC_K$ is quasiconstant and remains irreducible after pullback along any tame extension of $\kappa_K((\overline{t}))$.
Then for any $\calF \in \calC_K$ with $\calE \sim \calF$, 
for $\calF_0$ the regular component of $\calE^\dual \otimes \calF$,
the map $\calF_0 \otimes \calE \to \calF$ is an isomorphism.
In particular, $(\rank \calF)/(\rank \calE) = \rank \calF_0$ is a positive integer.
\end{lemma}
\begin{proof}
We may assume that $\kappa_K$ is perfect.
Our hypothesis on $\calE$ ensures that the regular component of $\calE^\dual \otimes \calE$ coincides with the trace component.
Consequently, it suffices to show that there exists some regular object $\calF_0$ for which there exists an isomorphism $\calF_0 \otimes \calE \to \calF$, as then we can recover the description of $\calF_0$ as the regular component of $\calE^\dual \otimes \calF$.

We may assume without loss of generality that there is a finite separable, totally tamely ramified 
extension $L$ of $\kappa_K((\overline{t}))$ such that the pullback of $\calE$ along $L$ is trivial. We induct on the degree of this extension, the case of the trivial extension being evident (as in this case $\calE$ is itself of rank 1).

By standard group theory (e.g., \cite[Exercise~3.2]{kedlaya-book}), we can find a subextension $L_0$ of $L/\kappa_K((\overline{t}))$ which is cyclic of degree $p$.
Let $f^*$ and $f_*$ be the pullback and pushforward functors for this extension.
By the induction hypothesis, we have an isomorphism $f^* \calF \cong \calF'_0 \otimes \calE'$
for some quasiconstant $\calE'$  and some regular object $\calF'_0$ (in the appropriate category); we may also identify $\calF'_0$ with the regular component of $(\calE')^\dual \otimes f^* \calF$. In particular, we have $\rank \calF = (\rank \calF'_0)(\rank \calE')$.
The following cases are mutually exclusive and exhaustive.
\begin{itemize}
\item
The wild monodromy representation of $\calE$ remains irreducible upon restriction to $L$.
In this case, $\calE' \cong f^* \calE$ and $\calF'_0$ contains $f^* \calF_0$. 
We thus have
\[
 (\rank \calF_0)(\rank \calE) \leq (\rank \calF'_0)(\rank \calE') = \rank \calF.
\]

\item
The wild monodromy representation of $\calE$ becomes reducible upon restriction to $L$.
In this case, $\calE \cong f_* \calE'$. In particular, if we choose a rank-1 quasiconstant object $\calL \in \calC_K$ corresponding to a nontrivial character of $\Gal(L_0/\kappa_K((\overline{t})))$, then $\bigoplus_{i=0}^{p-1} (\calL^{\otimes i} \otimes \calF_0)$
embeds into $\calE^\dual \otimes \calF$ and its pullback is contained in $\calF'_0$. We thus have
\[ (\rank \calF_0)(\rank \calE) \leq (\frac{1}{p} \rank \calF'_0)(p \rank \calE') = \rank \calF.
\]
\end{itemize}
In both cases, we deduce that the surjective map $\calF_0 \otimes \calE \to \calF$ must be an isomorphism.
\end{proof}

\begin{defn} \label{D:wild inertia rep}
Using Theorem~\ref{T:p-adic mono}, we may associate a \emph{wild monodromy representation}
to any $\calE \in \calC_K$ as follows. 
For $\calE$ irreducible, apply Proposition~\ref{P:quasiconstant rep} to choose a quasiconstant object $\calF \in \calC_K$
with $\calE \sim \calF$. We then include the wild monodromy representation associated to $\calF$ as a summand with multiplicity $(\rank \calE)/(\rank \calF)$, which is an integer by Remark~\ref{L:rank}.

By the same token, there exists a unique (up to isomorphism) quasiconstant object $\calF \in \calC_K$ such that $\calE$ and $\calF$ have isomorphic wild monodromy representations. We call $\calF$ a \emph{quasiconstant model} of $\calE$.
\end{defn}

\begin{remark}
Lemma~\ref{L:rank} gives us a way to view objects of $\calC_K$, up to a tame base extension,
as tensor products of regular objects with quasiconstant objects; this can be thought of as a loose $p$-adic analogue of the Turrittin--Levelt--Hukuhara decomposition theorem \cite[Theorem~7.5.1]{kedlaya-book}. Unfortunately, it seems extremely difficult to further analyze the structure of regular objects more closely without extra hypotheses (like the existence of a Frobenius structure). For some partial results, see the discussion of $p$-adic exponents in \cite[Chapter~13]{kedlaya-book}.
\end{remark}

\section{Relative connections on relative annuli}
\label{sec:relative conn}

In this section, we consider a relative version of the previous situation. Notably, we consider relative annuli over a base space $X$ and vector bundles with $\calO_X$-linear connections, rather than $K$-linear connections; in particular, we are still only differentiating in one direction, so integrability will play no role.

\begin{hypothesis}
Throughout \S\ref{sec:relative conn}, let $X$ be an adic space over $K$.
Let $I \subset (0, +\infty)$ be an interval (which if unspecified may be open, closed, or half-open).
Let $\calE$ be a vector bundle of rank $n$
on $X \times_K A_K[I]$ equipped with an (integrable) $\calO_X$-linear connection.
For $x \in X$, we write $\calE_x$ for the pullback connection on 
\[
x \times_K A_K[I] := \Spa(\calH(x), \calH(x)^\circ) \times_K A_K[I] \cong A_{\calH(x)}[I];
\]
by definition this depends only on the underlying point of $\overline{X}$.
\end{hypothesis}

\begin{remark}
Unless otherwise specified, we do not require $X$ to be locally tft (topologically of finite type).
We do rely on the fact that $X$ is covered by the spectra of \emph{sheafy} Huber rings,
but even this hypothesis could be relaxed by being more careful in formulation of some statements.
\end{remark}

\subsection{Relative connections}
\label{sec:relative}

We start with some relative versions of results from \cite[Chapter~11]{kedlaya-book}.

\begin{defn}
Suppose that $X$ is affinoid and $I$ is closed. A function $f\colon X \times I \to \RR$ is \emph{strongly subharmonic} if locally on $X$ and $I$,
it can be written in the form
\[
(x,\rho) \mapsto \max\{a_i |t_i|_{x,\rho}\colon i=1,\dots,n\}
\]
for some $a_i \in \RR_{\geq 0}$ and some $t_i \in \calO(X \times_K A_K[I])$, where $|t_i|_{x,\rho}$ denotes the $\rho$-Gauss norm on the restriction of $t_i$ to $x \times_K A_K[I]$;
any such function factors through a continuous function $\overline{X} \times I \to \RR$.
(To define a \emph{subharmonic} function we should allow certain limits; we will not need to clarify this here.)

The key example of this definition is the following. Let $P(T) \in \calO(X \times_K A_K[I])[T]$ be a monic polynomial and fix $i \in \{1,\dots,\deg P\}$. For $x \in X, \rho \in I$, 
let $\eta_{x,\rho}$ denote the $\rho$-Gauss point in $x \times_K A_K[I]$,
and let $F_i(x,\rho)$ be the absolute value of the product of the $i$ largest roots of $P$ in the field $\calH(\eta_{x,\rho})$. Then $F_i\colon X \times I \to \RR$ is computed by the Newton polygon of $P$ and thus is strongly subharmonic.
\end{defn}

\begin{lemma} \label{L:locally free}
Suppose that $X$ is affinoid and $I$ is closed,
and let $\calE$ be a vector bundle on $X \times_K A_K[I]$.
Then for every $x \in X$, there exists a partially proper open neighborhood $U$ of $X$ 
such that the restriction of $\calE$ to $U \times_K A_K[I]$ is free. More precisely,
given any basis of $x \times_K A_K[I]$, the supremum norm on $\Gamma(x \times_K A_K[I], \calE)$ defined by this basis coincides with the supremum norm defined by some basis of $U \times_K A_K[I]$ for some $U$.
\end{lemma}
\begin{proof}
We may assume from the outset that $x$ is a height-1 point.
In the case where $X$ is a point, the statement of the lemma follows from the fact that
$\calO(A_K[I])$ is a principal ideal domain
(e.g., see \cite[Proposition~8.3.2]{kedlaya-book}).
We reduce the general case to this case as follows.

Put $R := \calO(X \times_K A_K[I])$ and let $R_x$ be the colimit of $\calO(U \times_K A_K[I])$ over all partially proper open neighborhoods $U$ of $x$ in $X$.
Since $\calO_{X,x}$ has dense image in $\calH(x)$ and $\calH(x)[t^{\pm}]$ is dense in $\calO(x \times_K A_K[I])$,
the map $R_x \to \calO(x \times_K A_K[I])$ has dense image. 

We next show that every element of $R_x$ that maps to a unit in 
$\calO(x \times_K A_K[I])$ is itself a unit (so in particular the kernel of the map is contained in the Jacobson radical of $R_x$).
Suppose that $f \in R_x$ maps to a unit in $\calO(x \times_K A_K[I])$. Since $R_x$ has dense image in $\calO(x \times_K A_K[I])$, we may choose $g \in R_x$ which is a sufficiently good approximation of the inverse of $f$ in $\calO(x \times_K A_K[I])$ so as to ensure that $|fg-1| < 1$ in $\calO(x \times_K A_K[I])$.
Write $fg = 1 + \sum_i h_i t^i$ where $h_i \in \calO(U)$ for some partially proper open neighborhood $U$ of $x$ in $X$ (chosen independently of $i$).
Choose any $\epsilon  \in (0,1)$; then there exists a finite subset $S$ of $\ZZ$ such that for all $i \notin S$, $|h_i t^i| \leq \epsilon$ in $\calO(U \times_K A_K[I])$. Meanwhile, for each $i \in S$, we can find a partially proper open neighborhood $U_i$ of $x$ in $U$ such that $|h_i t^i| \leq \epsilon$ in $\calO(U_i \times_K A_K[I])$. Now $|fg-1| \leq \epsilon$ in $\calO(U' \times_K A_K[I])$
for $U' = \bigcap_i U_i$, so $fg$ is invertible in $R_x$ (as then is $f$).

This then implies that if $a_1,\dots,a_m \in R_x$ generate the unit ideal in $\calO(x \times_K A_K[I])$, then they also generate the unit ideal in $R_x$. Namely, if $b_1,\dots,b_m \in \calO(x \times_K A_K[I])$ satisfy $a_1 b_1 + \cdots + a_m b_m = 1$, then
for any $\epsilon \in (0,1)$
we may approximate the $b_i$ with elements $b'_i \in R_x$ in such a way that $|a_1 b'_1 + \cdots + a_m b'_m - 1| \leq \epsilon$ in $\calO(x \times_K A_K[I])$. Hence
$a_1 b'_1 + \cdots + a_m b'_m$ is an element of $R_x$ mapping to a unit in $\calO(x \times_K A_K[I])$, and hence is already a unit in $R_x$.

We now return to the original question.
Since $X \times_K A_K[I]$ is connected whenever $X$ is, we may reduce to the case where $\calE$ has constant rank $n$.
Choose generators $\bv_1,\dots,\bv_m$ of  $M := \Gamma(X \times_K A_K[I], \calE)$ as an $R$-module
and let $\be_1,\dots,\be_n$ be a basis of $\Gamma(x \times_K A_K[I], \calE) = M \widehat{\otimes}_{\calO(X)} \calH(x)$ as an $\calO(x \times_K A_K[I])$-module.
We can then write $\be_j = \sum_i A_{ij} \bv_i$ 
and $\bv_j = \sum_i B_{ij} \be_i$ for some $A_{ij}, B_{ij} \in \calO(x \times_K A_K[I])$. 
Write $A_{ij} = C_{ij} + D_{ij}$ where $C_{ij}$ has entries in $R_x$, $D$ has entries in $\calO(x \times_K A_K[I])$, and $|D| < |B|^{-1}$
(where the norm of a matrix is the supremum of the norm of its entries). Set 
\[
\be'_j := \sum_i C_{ij} \bv_i = \be_j - \sum_i D_{ij} \bv_i =  \sum_i (1-BD)_{ij} \be_i.
\]
Since $|BD| < 1$, $1-BD$ is an invertible matrix over $\calO(x \times_K A_K[I])$, so $\be'_1,\dots,\be'_n$ also form a basis of 
$\Gamma(x \times_K A_K[I], \calE)$. Now  the maximal minors of the matrix $C$ generate the unit ideal in $\calO(x \times_K A_K[I])$, and hence also in $R_x$ by the previous paragraph. Hence $\be'_1,\dots,\be'_n$ form a basis of $M \otimes_R R_x$, yielding the desired result.
\end{proof}

\begin{remark}
A theorem of Bartenwerfer \cite{bartenwerfer} implies that the analogue of
Lemma~\ref{L:locally free} holds for vector bundles on $X \times_K A_K[I_1] \times_K \cdots \times_K A_K[I_n]$ for any closed intervals $I_1,\dots,I_n$. More precisely, Bartenwerfer's theorem is stated in the language of rigid analytic geometry, and so implicitly assumes that $X$ is tft; however, we may again argue as in the proof of Lemma~\ref{L:locally free} to bootstrap from the case where $X$ is a point to the general case without the tft hypothesis.
\end{remark}

\begin{lemma} \label{L:relative continuous}
For $i=1,\dots,n$, for any $\epsilon > 0$, the function $F_i\colon \overline{X} \times I \to \RR$ given by
\[
F_i(x, \rho) = \sum_{j=1}^i \max\{\epsilon, -\log s_i(\calE_x, \rho)\}
\]
is strongly subharmonic.
\end{lemma}
\begin{proof}
We argue as in \cite[Theorem~11.3.2]{kedlaya-book}.
Since the claim is local on $\overline{X}$ and $I$, we may work locally around some $x_0 \in X$, $\rho_0 \in I$;
to begin with, we may thus assume that $X$ is affinoid and $I$ is closed.
Using Frobenius pushforwards \cite[Theorem~10.5.1]{kedlaya-book}, we may reduce to the case
where $\epsilon > \frac{1}{p-1} \log p$.

Using Lemma~\ref{L:locally free}, we may further reduce to the case where
$\calE$ admits a basis whose restriction to the $\rho_0$-Gauss point over $x_0$ is 
a ``good basis'' in the sense of the proof of \cite[Lemma~11.5.1]{kedlaya-book};
that is, the matrix of action $N$ of $\frac{d}{dt}$ has the property that its singular values and its norms of eigenvalues match. By continuity, by shrinking $X$ and $I$ we may ensure that neither the singular values nor the norms of eigenvalues vary too much over $\overline{X} \times I$. By then applying \cite[Theorem~6.7.4]{kedlaya-book} as in the proof of \cite[Lemma~11.5.1]{kedlaya-book}, we deduce that the function $F_i(x, \rho)$ is computed
by the Newton polygon of a certain polynomial over $\calO(X \times_K A_K[I])$, namely the characteristic polynomial of $N$.
\end{proof}

\begin{cor} \label{C:IR relative continuous}
For $i=1,\dots,n$, the function $\overline{X} \times I \to \RR$ given by
\[
(x, \rho) \mapsto \sum_{j=1}^i -\log s_i(\calE_x,\rho)
\]
is continuous. In particular, the function $\overline{X} \times I \to \RR$ taking $(x, \rho)$ to $\IR(\calE_x, \rho)$ is continuous.
\end{cor}
\begin{proof}
Since the claim is local on $X$ and $I$, we may reduce to the case where $X$ is affinoid and $I$ is closed.
By applying Lemma~\ref{L:relative continuous} for every $\epsilon > 0$, we may deduce the claim.
\end{proof}

\begin{cor} \label{C:relative continuous}
Suppose that $X$ is affinoid and let $S$ be a \emph{boundary} of $X$ in the sense of  \cite[\S~2.4]{berkovich1}. 
Then for $x \in X$, $\rho \in I$, 
\[
\IR(\calE_x, \rho) \geq \inf_{y \in S} \{\IR(\calE_y, \rho)\}.
\]
\end{cor}
\begin{proof}
This is immediate from Lemma~\ref{L:relative continuous}.
\end{proof}

\subsection{Decompositions of relative connections}

We continue with some relative versions of results from \cite[Chapter~12]{kedlaya-book}.

\begin{defn}
For $\rho > 0$, let $F_\rho$ be the field of analytic elements over $K$ in the variable $t$,
i.e., the completion of $K(t)$ for the multiplicative extension of the $\rho$-Gauss norm on $K[t]$ \cite[Definition~9.4.1]{kedlaya-book}. We view $F_\rho$ as a differential field for the continuous derivation $\frac{d}{dt}$.

For $X$ affinoid, we view $\calO(X) \widehat{\otimes}_K F_\rho$ as a differential ring for the derivation $\frac{d}{dt}$. Note that for any interval $I$ containing $\rho$, for any morphism $Y \to X$ such that
$\calO(X) \to \calO(Y)$ is injective, within $\calO(Y) \widehat{\otimes}_K F_\rho$
we have the equality
\begin{equation} \label{eq:intersection annulus with fiber}
(\calO(X) \widehat{\otimes}_K F_\rho) \cap (\calO(Y \times_K A_K(I)) = \calO(X \times_K A_K(I)) 
\end{equation}
\end{defn}

\begin{lemma} \label{L:global slope decomposition analytic elements}
Suppose that $X$ is affinoid. Let $M$ be a finite projective differential module of rank $n$ over $\calO(X) \widehat{\otimes}_K F_\rho$ for the derivation $\frac{d}{dt}$.
Suppose that for some $i < n$, we have
 $s_i(M_x, \rho) > s_{i+1}(M_x,\rho)$ for all $x \in X$.
Then there exists a unique direct sum decomposition $M = M_1 \oplus M_2$
of differential modules with $\rank M_1 = i$ such that for all $x \in \overline{X}$,
\begin{align*}
s_j(M_{1,x}, \rho) &= s_j(M_x, \rho) \qquad (j=1,\dots,i) \\
s_j(M_{2,x}, \rho) &= s_{j-i}(M_x, \rho) \qquad (j=i+1,\dots,n).
\end{align*}
\end{lemma}
\begin{proof}
We may check the claim locally around some $x_0 \in X$. Using Frobenius descendants as in the proof of 
\cite[Lemma~12.3.2]{kedlaya-book}, we may reduce to the case where 
$s_i(M_{x_0}, \rho) > p^{-1/(p-1)}$;
by Corollary~\ref{C:IR relative continuous}, we may shrink $X$ to further ensure that
$s_i(M_x, \rho) > p^{-1/(p-1)}$ for all $x \in X$.

Note that the colimit
of $\calO(U) \widehat{\otimes}_K F_\rho$ over all neighborhoods $U$ of $x_0$ in $X$ is a local ring with residue field $\calH(x_0) \widehat{\otimes}_K F$ (i.e., the field of analytic elements over $\calH(x_0)$ in the variable $t$).
We may thus apply \cite[Theorem~5.4.2]{kedlaya-book} over the residue field and then Nakayama's lemma; in other words,
after shrinking $X$ we may assume that $M$ admits a cyclic vector $\bv$. Put $D = \frac{d}{dt}$
and write $D^n(\bv) = a_0 \bv + a_1 D(\bv) + \cdots + a_{n-1} D^{n-1}(\bv)$ with 
$a_0,\dots,a_{n-1} \in \calO(X) \widehat{\otimes}_K F_\rho$.
We may now obtain the desired factorization by applying \cite[Theorem~2.2.2]{kedlaya-book} 
to the polynomial $T^n - a_0 - a_1 T - \dots - a_{n-1} T^{n-1}$ in the twisted polynomial ring
over $\calO(X) \widehat{\otimes}_K F_\rho$,
as in the proof of \cite[Theorem~6.6.1]{kedlaya-book}.
\end{proof}

\begin{lemma} \label{L:global slope decomposition}
Suppose that $I$ is open and that for some $i<n$, the following conditions hold for each $x \in \overline{X}$.
\begin{enumerate}
\item[(a)]
The function $r \mapsto \log s_1(\calE_x, e^{-r}) + \cdots +\log s_i(\calE_x, e^{-r})$ is affine for 
$r \in -\log I$.
\item[(b)]
We have $s_i(\calE_x, \rho) > s_{i+1}(\calE_x,\rho)$ for $\rho \in I$.
\end{enumerate}
Then there exists a unique direct sum decomposition $\calE = \calE_1 \oplus \calE_2$
of connections with $\rank \calE_1 = i$ such that for all $x \in \overline{X}$, $\rho \in I$,
\begin{align*}
s_j(\calE_{1,x}, \rho) &= s_j(\calE_x, \rho) \qquad (j=1,\dots,i) \\
s_j(\calE_{2,x}, \rho) &= s_{j-i}(\calE_x, \rho) \qquad (j=i+1,\dots,n).
\end{align*}
\end{lemma}
\begin{proof}
By \cite[Theorem~12.5.2]{kedlaya-book}, the claim holds when $X = \{x\}$;
combining this with Lemma~\ref{L:global slope decomposition analytic elements} using \eqref{eq:intersection annulus with fiber} yields the claim.
\end{proof}

\subsection{Regular relative connections}

We continue with some relative versions of results from \cite[Chapter~13]{kedlaya-book}.

\begin{defn}
Suppose that $I$ is open. We say that $\calE$ is \emph{regular} if $\calE_x$ is regular for each $x \in X$.
\end{defn}

\begin{example} \label{exa:variable exponent}
Let $R$ be the ring of continuous functions $\ZZ_p^n \to K$, equipped with the compact-open topology, and put $X := \Spa(R,R^\circ)$. 
The connected components of $X$ correspond to elements of $\ZZ_p^n$.

Let $\iota_1,\dots,\iota_n \in R$ be the elements corresponding to the projection maps $\ZZ_p^n \to \ZZ_p \to K$. Let $\calE$ be the connection on $X \times_K A_K(0,1)$ which is free on the basis $\be_1,\dots,\be_n$ for which
\[
\nabla(\be_i) = \be_i \otimes \iota_i \frac{dt}{t};
\] 
then $\calE$ is regular \cite[Example~9.5.2]{kedlaya-book}.
For $x \in X$ in the connected component corresponding to $A \in \ZZ_p^n$,
$A$ is an exponent for $\calE_x$.
\end{example}

Example~\ref{exa:variable exponent} does not rule out the possibility that one can find constant exponents for regular connections whenever $X$ is connected. We make a more modest conjecture here.
\begin{conj} \label{conj:finite Shilov boundary exponent}
Suppose that $I$ is open, $X$ is connected and affinoid with a finite Shilov boundary (the latter condition is automatic for $X$ a classical affinoid),
and $\calE$ is regular of rank $n$.
Suppose that there exist some $x \in X$, some closed subinterval $J$ of $I$ of positive length,
and some $A \in \ZZ_p^n$ with $p$-adic non-Liouville differences such that $A$ is an exponent for $\calE_x$ on $J$. Then for all $x \in X$, $A$ is an exponent for $\calE_x$.
\end{conj}

We prove some partial results towards Conjecture~\ref{conj:finite Shilov boundary exponent}.
A key point is the following ``bridge lemma''.
\begin{lemma} \label{L:bridge}
Suppose that $I$ is open and $\calE$ is regular of rank $n$.
Suppose that for some $x,y \in X$ and some closed subinterval $J$ of $I$ with nonempty interior,
\begin{enumerate}
\item[(a)]
there exists $A \in \ZZ_p^n$ with $p$-adic non-Liouville differences which is an exponent for $\calE_x$ on $J$;
\item[(b)]
there exists $B \in \ZZ_p^n$ which is an exponent for both $\calE_x$ and $\calE_y$ on $J$.
\end{enumerate}
Then $A$ is an exponent for both $\calE_x$ and $\calE_y$.
\end{lemma}
\begin{proof}
By Theorem~\ref{T:CM}(b), $A$ and $B$ are weakly equivalent, and hence equivalent since
$A$ has $p$-adic non-Liouville differences.
Hence $A$ is also an exponent for $\calE_{y}$ on $J$.
By Corollary~\ref{C:extend exponent}, $A$ is an exponent for both $\calE_x$ and $\calE_{y}$.
\end{proof}

\begin{lemma} \label{L:projection to uniform}
Suppose that $I$ is open, $X$ is connected and affinoid,
and $\calE$ is regular of rank $n$.
Let $S$ be a subset of $X$ for which there exists a bounded map $\lambda\colon \calO(X) \to V$ of $K$-Banach spaces with $1 \notin \ker(\lambda)$ which is submetric with respect to each element of $S$.
Then for every closed subinterval $J$ of $I$ such that $\Gamma(X \times_K A_K[J], \calE)$ is a free
module over $\calO(X \times_K A_K[J])$,
there exists $A \in \ZZ_p^n$ such that for all $x \in S$, 
$A$ is an exponent for $\calE_x$ on $J$.
\end{lemma}
\begin{proof}
This follows from a careful reading of the proof of \cite[Theorem~13.5.5]{kedlaya-book}.
By hypothesis, $\calE$ admits a basis $\be_1,\dots,\be_n$ on $X \times_K A_K[J]$.
For $m=0,1,\dots$ and $A \in \ZZ_p^n$, define the matrices $S_{m,A}$ over $\calO(X \times_K A_K[J])$
as in the proof of \cite[Theorem~13.5.5]{kedlaya-book}; note that $S_{m,A}$ depends only on the class
of $A$ in $(\ZZ/p^m \ZZ)^n$.
These matrices have the property that $\det(S_{0,A}) = 1$ and
\[
\det(S_{m+1,A}) = \sum_B \det(S_{m,B})
\]
where $B$ runs over a set of coset representatives for $A + p^{m+1} \ZZ_p^n$ in $A + p^m \ZZ_p^n$.
In particular, we may choose $A$ so that the constant terms of $\det(S_{m,A})$ have images in $V$ whose
norms are bounded away from 0. By our assumption on $S$, we deduce that 
$A$ is an exponent for $\calE_x$ on $J$ for each $x \in S$.
\end{proof}

\begin{lemma} \label{L:move to boundary1}
Suppose that $I$ is open, $X$ is connected and affinoid with a one-point boundary $\{x_0\}$,
and $\calE$ is regular of rank $n$ and admits a basis.
Then for every closed subinterval $J$ of $I$ and every $x \in X$,
there exists $A \in \ZZ_p^n$ such that $A$ is an exponent for both $\calE_{x_0}$ and $\calE_x$ on $J$.
\end{lemma}
\begin{proof}
Apply Lemma~\ref{L:projection to uniform} with $S = \{x_0, x\}$ and $\lambda\colon \calO(X) \to \calH(x)$ the natural ring homomorphism.
\end{proof}

\begin{cor} \label{C:move to boundary1}
Conjecture~\ref{conj:finite Shilov boundary exponent} holds in all cases where $X$ is connected and affinoid with a one-point boundary $\{x_0\}$ (but is not necessarily tft) and $\calE$ admits a basis.
\end{cor}
\begin{proof}
With notation as in the hypotheses of Conjecture~\ref{conj:finite Shilov boundary exponent},
we may apply Lemma~\ref{L:move to boundary1}
to produce $B \in \ZZ_p^n$ such that $B$ is an exponent for both $\calE_{x_0}$ and $\calE_x$ on $J$.
By Lemma~\ref{L:bridge}, $A$ is an exponent for $\calE_{x_0}$.

For any closed subinterval $J'$ of $I$ and any $x' \in X$, 
we may apply Lemma~\ref{L:move to boundary1} again to 
to produce $B' \in \ZZ_p^n$ such that $B'$ is an exponent for both $\calE_{x_0}$ and $\calE_{x'}$ on $J$.
By Lemma~\ref{L:bridge}, $A$ is an exponent for $\calE_{x'}$.
\end{proof}

\begin{lemma} \label{L:move to boundary annulus}
Conjecture~\ref{conj:finite Shilov boundary exponent} holds in all cases where $X$ is a finite \'etale cover of a closed annulus over $K$ and $\calE$ admits a basis.
\end{lemma}
\begin{proof}
Let $S$ be the skeleton of $X$.
In case $X = A_K[J]$ for some $J$, then the constant coefficient map defines a bounded $K$-linear functional on $\calO(X)$ which is submetric with respect to every $x \in S$; in the general case,
we obtain such a functional by first taking the trace down to an annulus and then multiplying by a suitably small scalar. 
We may thus apply Lemma~\ref{L:projection to uniform} to deduce that for every closed subinterval $J$ of $I$,
there exists $B \in \ZZ_p^n$ such that for all $x \in X$, 
$B$ is an exponent for $\calE_x$ on $I$.

Now set notation as in the hypotheses of Conjecture~\ref{conj:finite Shilov boundary exponent}.
Then there exists a unique point $x_0 \in S$ which dominates $x$, and Corollary~\ref{C:move to boundary1} implies that $A$ is an exponent for $\calE_{x_0}$. 
By Lemma~\ref{L:bridge}, $A$ is an exponent for $\calE_{y}$ for each $y \in S$.
By Corollary~\ref{C:move to boundary1} once more, $A$ is an exponent for $\calE_y$ for each $y \in X$.
\end{proof}

\begin{theorem} \label{T:unipotent spread}
Conjecture~\ref{conj:finite Shilov boundary exponent} holds in all cases where $X$ is tft and $\calE$ admits a basis.
\end{theorem}
\begin{proof}
We may assume that $K$ is algebraically closed.
We first treat the case where $X$ is of dimension 1; by pulling back from $X$ to its normalization, we may also assume that $X$ is smooth.
Using Corollary~\ref{C:move to boundary1} and a further base extension, we may reduce to the case where $x$ is a $K$-rational point. There is a unique point $x_1$ in the skeleton of $S$ which dominates $x$;
by Corollary~\ref{C:move to boundary1}, $A$ is an exponent for $\calE_{x_1}$. By repeated application of
Lemma~\ref{L:move to boundary annulus}, we deduce that $A$ is an exponent for $\calE_y$ for every $y$ in the skeleton of $S$. By Corollary~\ref{C:move to boundary1} again, we deduce the same for every $y \in X$.

In the general case, we may again assume that $x$ is a $K$-rational point.
For every complete extension $L$ of $K$ and every point $y \in X(L)$,
we may join $x$ and $y$ with a chain of curves in $X \times_K L$
(it suffices to check this when $x$ and $y$ map to the same affinoid subspace of $X$,
in which case we can find a one-dimensional complete intersection in $X$ containing both $x$ and $y$)
and then apply the previous argument to deduce that $A$ is an exponent for $\calE_y$. We may then use Corollary~\ref{C:move to boundary1} again to deduce the claim for arbitrary $y \in X$.
\end{proof}

\subsection{Solvable relative connections}

We next introduce the relative analogue of the category $\calC_K$.

\begin{defn}
For $X$ quasicompact, let 
$\calC_X$ be the 2-colimit over all $\epsilon > 0$ 
of the category of $\calO_X$-linear connections $\calE$ on
$X \times_K A_K(\epsilon,1)$ for which $\calE_x$ is solvable for each $x \in X$.
For general $X$, we define the corresponding category by glueing; the point is that an object of $\calC_X$ need not be defined over $X \times_K A_K(\epsilon,1)$ for any particular value of $\epsilon$.

The construction of $\calC_X$ is contravariantly functorial in $X$; in particular, for each $x \in X$
we have a pullback functor $\calC_X \to \calC_x$. We say that $\calE \in \calC_X$ is \emph{regular}
if $b(\calE_x) = 0$ for all $x \in X$.
\end{defn}

\begin{lemma} \label{L:Frobenius to solvable}
Assume that $X$ is quasicompact and that $I = (\epsilon,1)$ for some $\epsilon \in (0,1)$.
Let $q$ be a power of $p$.
Let $\varphi\colon X \times_K A_K[I^{1/q}] \to X \times_K A_K[I]$ be a (not necessarily $K$-linear) morphism
such that $|\varphi(t)-t^q|_\rho < 1$ for $\rho \in (\epsilon^{1/q},1)$,
and suppose that there exists an isomorphism $\varphi^*\calE \cong \calE$ of $\calO_X$-linear connections over $X \times_K A_K[I^{1/q}]$.
Then $\calE \in \calC_X$.
\end{lemma}
\begin{proof}
Define $\IR(\calE, \rho) = \inf_{x \in X} \{\IR(\calE_x, \rho)\}$.
As in \cite[Theorem~17.2.1]{kedlaya-book}, we have
\[
\IR(\calE, \rho^{1/q}) \geq \min\{\IR(\calE, \rho)^{1/q}, q \IR(\calE, \rho)\}.
\]
It follows that for any $\rho \in (0,1)$, $\IR(\calE, \rho^{1/q^n}) \to 1^-$ as $n \to \infty$.
For each $x \in X$, we have $\IR(\calE_x, \rho^{1/q^m}) \to 1^-$ as $m \to \infty$;
we may then apply Theorem~\ref{T:convexity} to deduce that $\calE_x$ is solvable.
\end{proof}

\begin{hypothesis} \label{H:solvable relative}
For the remainder of \S\ref{sec:relative conn}, assume that $\calE \in \calC_X$.
\end{hypothesis}

\begin{lemma} \label{L:sup function}
Suppose that $X$ is quasicompact.
For $i=1,\dots,n$, the following statements hold.
\begin{enumerate}
\item[(a)]
The quantity
\[
F_i(r) := \sup\left\{ \sum_{j=1}^i -\log s_i(\calE_x, e^{-r})\colon x \in X\right\}
\]
is finite.
\item[(b)]
The function $F_i\colon (0, -\log \epsilon] \to \RR$ is convex with nonnegative slopes and tends to $0$ as $r \to 0^+$.
\item[(c)]
The function $x \mapsto b_i(\calE_x)$ assumes only finitely many distinct values.
\end{enumerate}
\end{lemma}
\begin{proof}
We deduce (a) directly from Lemma~\ref{L:relative continuous}.
To deduce (b), view $F_i$ as the supremum of the functions $r \mapsto F_i(x, r) := \sum_{j=1}^i - \log s_i(\calE_x, e^{-r})$ over all $x \in X$.
 By Theorem~\ref{T:convexity}, each of these functions is convex with nonnegative slopes, as then is $F_i$. Since $F_i(x, r) \to 0$ as $r \to 0^+$ (by our assumption that $\calE$ is solvable), we also have 
\[
F_i(x,r) \leq \frac{r}{-\log \epsilon} F_i(x, -\log \epsilon) \leq \frac{r}{-\log \epsilon} F_i(-\log \epsilon),
\]
from which we conclude that $F_i(r) \to 0$ as $r \to 0^+$.
Similarly, $F_i(-\log \epsilon)/(-\log \epsilon)$ is an upper bound on the least slope of $F_i(r,x)$, so the function $x \mapsto b_1(\calE_x) + \cdots + b_i(\calE_x)$ assumes only finitely many distinct values; this yields (c).
\end{proof}

\begin{theorem} \label{T:break semicontinuity}
The following statements hold.
\begin{enumerate}
\item[(a)]
For $i=1,\dots,n$, the function $x \mapsto b_1(\calE_x) + \cdots + b_i(\calE_x)$ on $\overline{X}$ is upper semicontinuous.
In particular, $\{x \in X\colon b(\calE_x) = 0\}$ is open and partially proper.
\item[(b)]
Suppose that $X$ is affinoid and $S$ is a closed subset of $\overline{X}$ which is a boundary.
Then the supremum of $b(\calE_x)$ over all $x \in \overline{X}$ is achieved by some $x \in S$.
\item[(c)]
Suppose that $\calE$ is regular. 
Then for some $\epsilon' \in [\epsilon, 1)$,
we have $\IR(\calE_x, \rho) = 1$ for all $x \in X, \rho \in [\epsilon',1)$.
\end{enumerate}
\end{theorem}
\begin{proof}
We may assume throughout that $X$ is affinoid.
To prove (a), it will suffice to show that for each $c \in \frac{1}{n!} \ZZ$,
the set $U$ of $x \in \overline{X}$ for which $b_1(\calE_x) + \cdots + b_i(\calE_x) > c$ is closed.
By Lemma~\ref{L:criterion for big slope}, $U$ can be viewed as an intersection of sets,
one defined by the inequality in \eqref{eq:criterion for big slope} for each $\rho$;
by Corollary~\ref{C:IR relative continuous}, each of those inequalities defines a closed condition.

We next address (b). By Lemma~\ref{L:sup function}, the supremum $b$ is achieved by some
$x_0 \in X$. 
For $\rho \in [\epsilon,1)$, let $S_\rho$ be the subset of $x \in S$ such that $s_1(\calE_x, \rho) \geq \rho^b$; by comparison with $x_0$, we see that $S_\rho$ is nonempty for all $\rho$.
On the other hand, by convexity (Theorem~\ref{T:convexity}) we have $S_\rho \subseteq S_{\rho'}$
whenever $\rho \in [\rho',1)$. Since $\overline{X}$ is compact, we conclude that $\bigcap_\rho S_\rho$ is nonempty, and any $x$ in the intersection has the property that $b(\calE_x) = b$.

To prove (c), it suffices to check that for $x_0 \in X$, the claim holds after replacing $X$ with some neighborhood of $x_0$.
Since $\calE_{x_0}$ is regular, we may choose $\epsilon_1 \in [\epsilon,1)$ such that
$s_1(\calE_{x_0}, \rho) = 1$ for $\rho \in [\epsilon_1, 1)$.
Choose some $\delta \in (\epsilon_1^{1/n!},1)$;
by Corollary~\ref{C:IR relative continuous},
we may replace $X$ with a neighborhood of $x_0$ so as 
to ensure that $s_1(\calE_x, \epsilon_1) > \delta$ for all
$x \in X$.
By Theorem~\ref{T:convexity},
for each $x \in X$ the function $r \mapsto -\log s_1(\calE_x, e^{-r})$
is continuous and convex with slopes in $\frac{1}{n!} \ZZ$;
in particular, none of these slopes lie in the interval $(0, 1/n!)$. 
Consequently, the claim holds for $\epsilon' = \epsilon_1/\delta^{n!}$.
\end{proof}

\subsection{The relative monodromy theorem}

We finally formulate the relative version of the local monodromy theorem.

\begin{defn}
We say that a finite \'etale cover $Y$ of $X \times_K A_K(\epsilon,1)$ is \emph{eligible}
if for each $x \in X$, the pullback cover of $x \times_K A_K(\epsilon,1)$ is eligible.
\end{defn}

\begin{lemma} \label{L:spread eligible cover}
Choose $x_0 \in X$ and let $Y_0$ be an eligible finite \'etale cover of $x_0 \times_K A_K(\epsilon,1)$. Then there exist a neighborhood $U$ of $x_0$ in $X$, a value $\epsilon' \in (\epsilon,1)$, and an eligible finite \'etale cover $Y$ of $U \times_K A_K(\epsilon',1)$
whose restriction to $x_0 \times_K A_K(\epsilon',1)$ is isomorphic to $Y_0 \times_{A_K(\epsilon,1)} A_K(\epsilon',1)$.
\end{lemma}
\begin{proof}
We may assume at once that $Y_0$ is connected.
Let $L$ be the finite \'etale extension of $\kappa_{\calH(x_0)}((\overline{t}))$ giving rise to $Y \times_X x_0$. By induction on the degree of $L$ over $\kappa_{\calH(x_0)}((\overline{t}))$,
it suffices to treat the cases where this extension is unramified, totally tamely ramified, or wild and cyclic of degree $p$.

\begin{itemize}
\item
In the unramified case, $L$ is generated by an unramified extension of $\kappa_{\calH(x_0)}$. 
We may lift this to an unramified extension of $\calH(x_0)$ and then spread it out over an open (and partially proper) neighborhood of $x_0$.

\item
In the tamely ramified case, $L$ is a Kummer extension, which we may again lift.

\item
In the wild cyclic case, $L$ is generated by an element $z$ for which $z^p-z \in t^{-1} \kappa_{\calH(x_0)}[t^{-1}]$. By rescaling $t$ suitably, we may force this polynomial to have coefficients in the valuation subring of $\kappa_{\calH(x_0)}$; we may then again lift.
(Note that this case does not yield a partially proper neighborhood; this is why we use Huber spaces instead of Berkovich spaces throughout.)\qedhere
\end{itemize}
\end{proof}

\begin{defn}\label{D:quasiconstant model}
For $\calE \in \calC_X$ (per Hypothesis~\ref{H:solvable relative}),
a \emph{quasiconstant model} of $\calE$ is an object $\calF \in \calC_X$
such that:
\begin{itemize}
\item[(a)]
$\calF_x$ is a quasiconstant model of $\calE_x$ for each $x \in X$;
\item[(b)]
for some $\epsilon \in (0,1)$,
there exists an eligible finite \'etale cover of $X \times_K A_K(\epsilon,1)$ 
which at each $x \in X$ specializes to a cover which makes $\calF_x$ trivial (after suitable
extension of the constant field).
\end{itemize}
\end{defn}

\begin{remark}
If $X$ is smooth over $K$, then condition (b) in Definition~\ref{D:quasiconstant model} allows us to promote $\calF$ to a $K$-linear integrable connection
on $X \times_K A_K(\epsilon,1)$. One consequence of this is that if $\calF'$ is another quasiconstant model, then the local horizontal sections of $\calF^\dual \otimes \calF'$ form an integrable connection on $X$. In particular, if $X$ is affinoid then $\calF$ and $\calF'$ are isomorphic as relative connections, but \emph{not canonically}.
\end{remark}

\begin{theorem} \label{T:quasiconstant pullback}
For $\calE \in \calC_X$, for each $x_0 \in X$, 
there exist an affinoid subspace $U$ of $X$ containing $x_0$
and a quasiconstant model of $\calE|_U \in \calC_U$.
\end{theorem}
\begin{proof}
By Theorem~\ref{T:p-adic mono}, the claim holds when $X = \{x_0\}$.
To prove the general case, we may spread out the resulting eligible cover using Lemma~\ref{L:spread eligible cover} and then pull back
to reduce to the case where $\calE_{x_0}$ is regular.
In this case, Theorem~\ref{T:break semicontinuity} implies that
there exists some choice of $U$ for which $\calE_x$ is regular for all $x \in U$;
this proves the claim.
\end{proof}

\begin{cor} \label{C:regular from fibers}
If $\calE_x$ is regular for each $x$ in some dense subset of $X$, then $\calE$ is regular.
\end{cor}
\begin{proof}
We may work locally on $X$, so by Theorem~\ref{T:quasiconstant pullback} we may assume that
$\calE$ admits a quasiconstant model, and then replace $\calE$ with said model. By induction on the minimum degree of the eligible finite cover, we may further reduce to the case that $\calE$ becomes fiberwise constant after a $\ZZ/p\ZZ$-extension. In this setting, $\calE$ decomposes as a direct sum of line bundles, so we may assume that $n = 1$. 
We may assume that $K$ contains an element $\pi$ with $\pi^{p-1} = -p$.

Suppose by way of contradiction that $\calE$ is not regular. Then $\calE$ is free on a single generator $\bv$ satisfying $\nabla(\bv) = \pi \sum_{i=1}^m f_i t^{-i-1}\,dt$ where $f_i \in \calO(X)$, $|f_i| \leq 1$ for all $i$, and $|f_m| = 1$. (The eligible cover is then the Artin--Schreier cover $z^p - z = \sum_i \overline{f}_i t^i$.) By replacing $X$ with a suitable finite flat cover of $X$, we may further ensure that $m$ is not divisible by $p$.
Now let $U$ be the set of $x \in X$ for which $|f_m|_x \geq 1$; then $U$ is an open subset of $X$ with the property that $b_1(\calE_x) = m$ for all $x \in U$. However, this contradicts the hypothesis that $\calE_x$ is regular for each $x$ in some dense subset of $X$.
\end{proof}

We illustrate the limitations of Theorem~\ref{T:quasiconstant pullback} with an example derived from the classical Bessel equation.
\begin{example}
Assume that $p \neq 2$ and that $K$ contains an element $\pi$ with $\pi^{p-1} = -p$.
Let $\calE_0$ be 
the $K$-linear connection on $X \times_K A_K(0,1)$ defined by the matrix
\[
\begin{pmatrix}
0 & t^{-1}\,dt \\
\pi^2 t^{-2}\,dt & 0 \end{pmatrix}.
\]
As explained in \cite[Example~20.2.1]{kedlaya-book} (following \cite[Example~6.2.6]{tsuzuki}), this becomes regular after pullback along the extension
\[
L_0 = \kappa_K((\overline{t}))[\overline{t}^{1/2}, z]/(z^p - z - \overline{t}^{-1/2}).
\]
Now let $X$ be the closed unit disc in the coordinate $x$ and let $\calE$ be the pullback of $\calE_0$ along the map $(x,t) \mapsto x^{-1} t$. Let $\eta \in X$ be the Gauss point; then
$\calE_\eta$ becomes regular after pullback along the extension
\[
L = \kappa_K(\overline{x})((\overline{t}))[\overline{x}^{1/2} \overline{t}^{1/2}, z]/(z^p - z - \overline{x}^{1/2} \overline{t}^{-1/2}),
\]
which does not extend to a finite \'etale cover of $\kappa_K[x]((t))$.

Note that in this example, the relative connection does not extend to a $K$-linear connection unless we allow a logarithmic singularity along $x=0$. It may be that one can obtain better results for $K$-linear connections; we will not pursue this point here.
\end{example}

\section{Applications}

We conclude with some applications of the relative $p$-adic local monodromy theorem.

\subsection{Semistable reduction for isocrystals}

\begin{defn} \label{D:regular integrable}
Let $W$ be an adic space over $K$ with a one-point boundary $\eta$.
Consider the space $X := W \times_K A_K[\epsilon,1)$; without referring to the first projection map, we cannot describe the fibers of $X$ over points of $W$. However, the fiber $X_\eta$ can be described using only the second projection: for each $\rho \in [\epsilon,1)$, the fiber of $t_\rho$ for the second projection admits a unique one-point boundary, and this point belongs to $X_\eta$.

Now suppose that $W$ is smooth over $K$ and $\calE$ is an integrable $K$-linear connection on $X$. Then for any choice of the first projection $X \to W$, we may view $\calE$ as an object of $\calC_W$ and ask whether or not the resulting object is regular. By Theorem~\ref{T:break semicontinuity} this depends only on the restriction to $X_\eta$,
where the condition can be phrased in a manner that does not refer to the first projection:
namely, we want that for $\rho$ sufficiently close to 1, the restriction of the connection to the disc $|t-t_\rho| < \rho$ is isomorphic to the fiber at $t_\rho$ (further equipped with
the action of derivations that kill $t$).
Consequently, the regularity property is independent of the choice of the first projection.
\end{defn}

\begin{defn}
We may associate to every smooth scheme $X$ over $k$ the category of \emph{overconvergent isocrystals on $X$ with coefficients in $K$},
e.g., by considering crystals on the overconvergent site of Le Stum \cite{lestum}. 
To make this concrete, consider an open affine subscheme $U$ of $X$ 
and an open immersion $U \to \overline{U}$ of affine $k$-schemes whose complement $Z$ is the smooth integral $k$-scheme cut out by some $t \in \overline{U}$.
Fix a smooth formal scheme $P$ over $\frako_K$ lifting $Z$ and let $W$ be the Raynaud generic fiber of $P$. Then an overconvergent isocrystal $\calE$ on $X$ with coefficients in $K$ restricts 
to an integrable $K$-linear connection on $W \times_K A_K[\epsilon,1)$ for some $\epsilon \in (0,1)$. Using Definition~\ref{D:regular integrable}, we may define
the condition that $\calE$ is \emph{regular along $Z$}. Note that we may keep track of the choice of $Z$ in terms of its associated divisorial valuation on the function field $k(U)$.
\end{defn}

\begin{theorem} \label{T:global semistable}
Let $X$ be a smooth scheme over $k := \kappa_K$.
Let $\calE$ be an overconvergent isocrystal on $X$ with coefficients in $K$.
Then there exist a dominant, generically finite morphism $f\colon Y \to X$ such that for any dominant morphism $g\colon U \to Y$ with $U$
smooth and any open immersion $j\colon U \to \overline{U}$ with $\overline{U}$ smooth over $k$ and 
$\overline{U} \setminus U$ a strict normal crossings divisor, $(g \circ f)^* \calE$ is regular along each component of $\overline{U} \setminus Y$.
\end{theorem}
\begin{proof}
We may check the claim after replacing $X$ with an open dense subscheme or pushing forward along a finite \'etale morphism. Since $X$ is covered by open dense subschemes each of which can be written as a finite \'etale cover of an affine space \cite{kedlaya-etale}, we may reduce to the case $X = \AAA^n$. If we blow up $\PP^n$ at a point in $X$, we obtain a $\PP^1$-fibration $Y$ over $\PP^{n-1}$; we now obtain from $\calE$ a $K$-linear integrable connection on an annulus bundle $\tilde{Y}$ over $\PP^{n-1,\ana}_K$.

We now apply Theorem~\ref{T:quasiconstant pullback}; this yields a collection of eligible finite covers, each defined over $U_i \times_K A_K[\epsilon,1)$ for some open subset $U_i$ of $\PP^{n-1,\ana}_K$ and some $\epsilon \in (0,1)$. Passing to Raynaud's point of view, the covering of $\PP^{n-1,\ana}_K$ by the $U_i$ arises from a Zariski covering of some admissible blowup of the formal completion of $\PP^{n-1}_{\mathfrak{o}_K}$ along $\PP^{n-1}_k$.
Let $P$ be the special fiber of the blowup, so that each $U_i$ corresponds to some open affine subset $\Spec R_i$ of $P$, and the eligible cover associated to $U_i$ corresponds to a finite \'etale cover of $R_i((t))$; using the Katz--Gabber construction
\cite{katz-gabber} we may realize the latter as the pullback of a finite \'etale cover $g_i\colon V_i \to \Spec R_i[t^{-1}]$. We now construct $f$ by choosing a generically finite map to $Y \times_P \PP^{n-1}_k$ dominating each $g_i$; 
this has the desired effect.
\end{proof}

\begin{remark} \label{R:recover semistable reduction}
Theorem~\ref{T:global semistable} gives a new proof of the semistable reduction theorem for overconvergent $F$-isocrystals as stated in \cite[Theorem~2.4.4]{kedlaya-semi4}.
In that statement, $K$ is assumed to be discretely valued and $\calE$ is assumed to carry a Frobenius structure.

Let us spell out in more detail how this implication works.
In Theorem~\ref{T:global semistable}, by de Jong's alterations theorem \cite{dejong-alterations} the map $f$ can be chosen so that $Y$ itself admits a good compactification $\overline{Y}$. x
For each boundary component $Z$ of $\overline{Y} \setminus Y$, Theorem~\ref{T:p-adic mono with Frob} implies that after replacing $f$ with a further tame covering, $f^* \calE$ has unipotent monodromy along each $Z$.
We may then apply \cite[Proposition~6.3.2]{kedlaya-semi1}
to obtain an extension of $f^* \calE$ to a convergent log-$F$-isocrystal with nilpotent residues on $\overline{Y}$ for the canonical logarithmic structure defined by $Z$.

The method of proof of Theorem~\ref{T:global semistable} is flexible enough to make it easily adaptable to related situations. We leave such adaptations to the interested reader.
\end{remark}

The following question arises if one tries to prove Theorem~\ref{T:global semistable} without reduction to the case of projective space.
\begin{conj} \label{conj:multiple lmt}
Let $\calE$ be an integrable $K$-linear connection on $A_K[\epsilon,1)^m$ for some positive integer $m$ and some $\epsilon \in (0,1)$. Then for some $\epsilon' \in [\epsilon,1)$, there exists some finite \'etale covering $Y \to A_K[\epsilon',1)^m$ such that for every nonarchimedean field $L$ and every morphism $A_L[\epsilon'',1) \to Y \times_K L$ of adic spaces over $L$, the pullback of $\calE$ to $A_L[\epsilon'',1)$ is regular as an object of $\calC_L$.
More precisely, $Y \to A_K[\epsilon',1)^m$ is ``eligible'' in the sense of being induced by a finite \'etale ring extension of the bounded subring of $\calO(A_K[\epsilon',1)^m)$.
\end{conj}

\begin{remark}
By imitating the proof of Theorem~\ref{T:quasiconstant pullback},
one can deduce from Conjecture~\ref{conj:multiple lmt} a corresponding relative version: for $X$ an adic space over $K$ and $\calE$ an integrable $\calO_X$-linear connection on $X \times_K A_K[\epsilon,1)^m$, locally on $X$ there is an ``eligible'' finite \'etale covering $Y\to X \times_K A_K[\epsilon',1)^m$ with a similar effect.
\end{remark}

\subsection{de Rham local systems}

We spell out some remarks which are left implicit in \cite[\S 7]{kedlaya-liu2}.
\begin{remark}
Let $L$ be a complete discretely valued field of mixed characteristic $(0,p)$
with perfect residue field. We may use Fontaine's period rings to define the conditions \emph{crystalline} and \emph{semistable} on a continuous representation of $G_L$ on a finite-dimensional $\QQ_p$-vector space.
However, we use the term \emph{log-crystalline} in place of \emph{semistable} so as to avoid confusion with the semistable condition on vector bundles on the Fargues--Fontaine curve.
\end{remark}

\begin{defn}
For any perfectoid field $L$ of characteristic 0, let $\FaFo_L$ denote the Fargues--Fontaine curve (with coefficients in $\QQ_p$) associated to the field $L$; it admits a distinguished point with residue field $L$, which we call the \emph{de Rham point}.

There is a functorial ``Narasimhan--Seshadri'' correspondence between continuous representations of $G_L$ on finite-dimensional $\QQ_p$-vector spaces and vector bundles on $\FaFo_L$ which are semistable of degree 0. More generally, vector bundles which are semistable of degree $r/s$, where $r/s \in \QQ$ is written in lowest terms, correspond to continuous representations of $G_L$ on finite-dimensional $\QQ_{p^s}$-vector spaces.

A \emph{modification} of vector bundles on $\FaFo_L$ consists of a pair of vector bundles $V,V'$ together with a meromorphic map $V \to V'$ which is an isomorphism away from the de Rham point.
\end{defn}

\begin{defn}
For any nonarchimedean field $L$ containing $\QQ_p$, we may define $\FaFo_L$ in the category of diamonds, and define the category of vector bundles on it by descent from the case of a perfectoid field. We then obtain a fully faithful functor from the category of continuous representations of $G_L$ on finite-dimensional $\QQ_p$-vector spaces to the category of vector bundles on $\FaFo_L$, whose essential image consists of bundles whose pullbacks to the Fargues--Fontaine curve over any perfectoid field containing $L$ are semistable of degree 0.

Let $L$ be a complete discretely valued field of mixed characteristic $(0,p)$
with perfect residue field.
By the Beauville--Laszlo theorem, a representation of $G_L$ is de Rham if and only it corresponds to a vector bundle $V$ admitting a
modification $V \to V'$
such that the pullback of $V'$ to the formal completion at the de Rham point is trivial as a $G_L$-representation. This modification is unique if it exists; we call it the \emph{canonical modification} of $V$.
\end{defn}

\begin{lemma} \label{L:Berger log-crystalline}
Let $L$ be a complete discretely valued field of mixed characteristic $(0,p)$
with perfect residue field.
Let $\rho$ be a de Rham representation of $G_L$, let $V$ be the corresponding vector bundle on $\FaFo_L$, and let $V \to V'$ be the canonical modification.
Then $\rho$ is log-crystalline if and only if each successive quotient of the Harder--Narasimhan filtration of $V'$ corresponds to a trivial $G_L$-representation.
\end{lemma}
\begin{proof}
We deduce this from \cite[Th\'eor\`eme~3.6]{berger-mono} as follows.
Let $K_0$ be the maximal absolutely unramified subfield of $L(\mu_{p^\infty})$.
Let $M$ be the $(\varphi, \Gamma)$-module over the Robba ring 
\[
\calR_{K_0} = \varinjlim_{\epsilon \to 1^-} \calO(A_{K_0}[\epsilon,1))
\]
associated to $V'$. From the construction of the canonical modification, we see that the
action of $\Gamma$ induces an action of its Lie algebra with removable singularities at the zero locus of $t$. By \cite[Th\'eor\`eme~3.6]{berger-mono} and Lemma~\ref{L:unipotent to log}, this connection is unipotent if and only if $\rho$ is log-crystalline and the space of logarithmic horizontal sections carries a trivial $\Gamma_L$-action. Note that the connection on $M$ carries an action of Frobenius whose slope filtration corresponds to the Harder--Narasimhan filtration of $V'$; by Corollary~\ref{C:unipotent pure slope} the connection is unipotent if and only if the induced connection on each successive quotient of the slope filtration is trivial. This yields the desired result.
\end{proof}

\begin{defn} \label{D:andreatta-brinon}
Assume that $K$ is discretely valued and that $\kappa_K$ is perfect.
Let $X$ be an affinoid  space $X$ over $K$
which admits an \'etale morphism $\pi\colon X \to \GG_m^k$. 
For simplicity, we assume that $X$ is a finite \'etale cover of a rational subspace of some polycircle $|T_i| = 1$; the point is that this holds locally on $X$. (Namely, we may start with $X$ finite \'etale over a rational subspace of a polyannulus $\alpha_i \leq |T_i| \leq \beta_i$.
We then rescale to make $\beta_i < 1$ for all $i$; this polyannulus is then isomorphic via $T_i \mapsto T_i+1$ to a rational subspace of the polydisc $|T_i-1| \leq \beta_i$, which is contained in the polycircle $|T_i| = 1$.)

Let $K_\infty$ be the completion of $K(\mu_{p^\infty})$, which is a perfectoid field.
Let $X_n$ be the space obtained from $X$ by pulling back along the $p^n$-power map on $\GG_m^k$,
then base extending from $K$ to $K(\mu_{p^n})$. Let $\psi$ be the tower formed by the spaces $X_n$;
this is a \emph{relative toric tower} in the sense of \cite[\S 7.2]{kedlaya-liu2}.
In particular, it admits a tilde-inverse limit $\tilde{X}_\psi$ which is a perfectoid space;
this space (resp. its tilt) is a finite \'etale cover of a rational subspace of the perfectoid polycircle $|U_i| =1$ over $K_\infty$ (resp. over $K_\infty^\flat$).
The tower $\psi$ is Galois with Galois group $\Gamma = \ZZ_p^k \rtimes \Gamma_0$ where $\Gamma_0 = \Gal(K(\mu_{p^\infty})/K) \subseteq \ZZ_p^\times$, so $\Gamma$ acts on $\tilde{X}_\psi$; we fix the splitting of the semidirect product for which $\Gamma_0$ fixes the $U_i$.

Let $L$ be a $\QQ_p$-local system (for the pro-\'etale topology) on $X$.
We may then associate to $L$ a vector bundle $V$ on the relative Fargues--Fontaine curve
$\FaFo_{\tilde{X}_\psi}$ which is fiberwise (over $\tilde{X}_\psi$) semistable of degree 0
\cite[Theorem~9.3.13]{kedlaya-liu1}; since $V$ is functorial in $L$, it inherits an action of $\Gamma$.
The vector bundle $V$ can also be interpreted as a 
$(\varphi, \Gamma)$-module $\tilde{M}$ over the perfect period ring $\tilde{\mathbf{C}}_\psi$;
using the decompletion process described in \cite[\S 5.7]{kedlaya-liu2}, we may descend $\tilde{M}$ canonically to a $(\varphi, \Gamma)$-module $M$ over the imperfect period ring $\mathbf{C}_\psi$. 
(This generalizes the setting considered by Andreatta--Brinon \cite{andreatta-brinon}.)

To make this more explicit, let $K_0$ be the maximal unramified subextension of $K(\mu_{p^\infty})$. We claim that for some $\rho \in (0,1)$, for $r \in \QQ \cap (0,1)$ sufficiently close to 1,
there exist an affinoid space $X'$ over $K_0$
such that $\mathbf{C}^{[r/q,s]}_\psi = \calO(X' \times_{K_0} A_{K_0}[\rho^r, \rho^{r/q}])$
and $\varphi$ induces a map from $X' \times_{K_0} A_{K_0}[\rho^{r/q},\rho^{r/q^2}]$ to $X' \times_{K_0} A_{K_0}[\rho^{r},\rho^{r/q}]$.
To see this, we may first check the claim explicitly when $X$ is the polyannulus $|T_i| =  1$,
in which we may take $X'$ to also be a polycircle;
then observe that the claim is preserved by replacing $X$ with a rational localization or a finite \'etale covering.

To make this more precise, let $\psi_0$ be the tower consisting of the points $\Spa(K(\mu_{p^n}), K(\mu_{p^n})^\circ)$.
In case $K = \QQ_p$, the ring $\mathbf{C}_{\psi_0}$ is a ring of Laurent series in the variable $\pi$ and the action of $\varphi$ and $\Gamma$
is given by
\[
\varphi(1+\pi) = (1+\pi)^p, \qquad \gamma(1 + \pi) = (1+\pi)^{\gamma} \quad (\gamma \in \ZZ_p^\times).
\]
In general this ring is a subring of $\mathbf{C}_{\psi_0}$; in particular, $\mathbf{C}_{\psi_0}$ and $\mathbf{C}_{\psi}$ contain the element $t = \log (1+\pi)$
on which $\ZZ_p^\times$ acts via $\gamma(t) = \gamma t$ (i.e., via the cyclotomic character on $\Gamma_0$).

The action of $\Gamma$ on $\mathbf{C}_\psi$ induces an action of the Lie algebra $\Lie \Gamma$ by derivations. In particular, using the aforementioned splitting $\Gamma_0 \to \Gamma$, we obtain an action of $\Lie \Gamma_0$. The subring $R$ of $\mathbf{C}_\psi$ which is killed by this action is an affinoid algebra over $K_0$; put $X' := \Spa(R,R^\circ)$.

The action of the Lie algebra of $\Gamma$ provides $M$ with a $K_0$-linear connection,
compatible with $\varphi$, which is singular along the zero locus of $t \in \mathbf{C}_\psi$.
Notably, this connection is \emph{not integrable} because $\Gamma$ is not commutative.
Nonetheless, we may restrict to $\Lie \Gamma_0$ to get a singular $\calO_{X'}$-linear connection $\calE$ on $X' \times_{K_0} A_{K_0}[\rho^{r}, 1)$.
\end{defn}

\begin{theorem} \label{T:relative de Rham}
Assume that $K$ is discretely valued and that $\kappa_K$ is perfect.
Let $L$ be a de Rham $\QQ_p$-local system on a smooth (in particular tft) adic space $X$ over $K$.
Then for every $x \in X$, there exist an open neighborhood $U$ of $x$ in $X$ and a finite \'etale covering $V \to U$ such that the pullback of $L$ to every classical point of $V$ is log-crystalline (semistable).
\end{theorem}
\begin{proof}
Since the claim is local at $x$,
we may assume that $X$ is affinoid and admits an \'etale morphism $\pi\colon X \to \GG_m^k$ which factors as a composition of rational localizations, finite \'etale morphisms,
and the embedding of the polycircle $|T_i| = 1$.
We may then set notation as in Definition~\ref{D:andreatta-brinon}.

As in Berger's construction \cite{berger-mono}, the de Rham condition ensures that the singularities of $\calE$ are removable. Let $\calE'$ be the $\calO_{X'}$-linear connection obtained by removing the singularities;
by Lemma~\ref{L:Frobenius to solvable} this connection is solvable.
By applying Theorem~\ref{T:quasiconstant pullback} to $\calE'$, we may reduce
to the case where this connection is regular.
In this case, the restriction of the connection to any classical point of $X'$ admits a relative Frobenius structure and so by Lemma~\ref{L:regular Frobenius equals tame} is unipotent. By Theorem~\ref{T:unipotent spread}, the restriction of $\calE'$ to any point of $X'$ is unipotent. By Lemma~\ref{L:unipotent to log}, the horizontal sections of $\calE'[\log t]$ themselves form a vector bundle on $X'$ carrying an integrable connection induced by the action of $\Lie \ker(\Gamma \to \Gamma_0)$; we may thus deduce that $\calE'$ is globally unipotent, that is, $\calE'$ admits a filtration whose successive quotients are constant. On each of these successive quotients, $\Gamma_0$ acts through some finite quotient. By adjoining a suitable $p$-power root of unity to $K$, we can ensure that these actions of $\Gamma_0$ are in fact trivial;
Lemma~\ref{L:Berger log-crystalline} now implies that the pullback of $L$ to every classical point of $V$ is log-crystalline.
\end{proof}

\begin{remark}
We remind the reader that notwithstanding the last clause of \cite[Remark~1.4]{liu-zhu},
Theorem~\ref{T:relative de Rham} cannot be upgraded to assert that
the covering $V \to U$ is the base extension from $K$ to some finite extension;
see \cite{lawrence-li} or \cite[Example~7.8]{guo-yang}.

We also recall from the introduction that \cite[Remark~1.4]{liu-zhu} asked for something stronger than
Theorem~\ref{T:relative de Rham}: does there exist a single finite \'etale cover $Y \to X$ such that the pullback of $L$ to every classical point of $Y$ is log-crystalline? Resolving this question using our techniques would require establishing some canonicality for the local coverings so that they can be glued together.
\end{remark}

\begin{remark}
The conclusion of Theorem~\ref{T:relative de Rham} becomes much more powerful if we combine it with a recent theorem of Guo--Yang \cite[Theorem~1.1]{guo-yang}.
That result says that if $X$ admits a $p$-adic integral model $\mathfrak{X}$ and $L$ is a $\ZZ_p$-local system on $X$ whose pullback of $L$ to every classical point of $X$ is crystalline (resp. log-crystalline), then $L$ is in fact crystalline (resp. log-crystalline) with respect to the model $\mathfrak{X}$.
In other words, if we start with a de Rham $\QQ_p$-local system on $X$, then after pulling back along a suitable \'etale cover (to apply Theorem~\ref{T:relative de Rham} and also to choose local $\ZZ_p$-lattices), we obtain log-crystallinity with respect to any integral models that we can write down. Moreover, log-crystallinity with respect to a single integral model implies the same with respect to any other model.

Interestingly, the argument of Guo--Yang itself uses a statement to the effect that ``relative de Rham implies relative potentially log-crystalline'', but with respect to a ramified cover (see \cite[\S 6]{guo-yang}).
\end{remark}

\begin{remark} \label{R:not etale}
The proof of Theorem~\ref{T:relative de Rham} also applies in case we start with a vector bundle on the relative Fargues--Fontaine curve over $X$ in the sense of \cite{kedlaya-liu1}; such a bundle corresponds to a $\QQ_p$-local system if and only it is fiberwise semistable of degree $0$.
However, we do not know if there is an analogue of the result of Guo--Yang in that setting; it is not even immediately obvious what the statement should say.
\end{remark}

\begin{remark}
Xin Tong suggests that it should be possible to adapt the proof of Theorem~\ref{T:relative de Rham} to handle arithmetic families of Galois representations valued in an affinoid algebra $A$ over $\QQ_p$, as in \cite{kedlaya-liu-families}, or even $A$-vector bundles on the relative Fargues--Fontaine curve over $X$, as in \cite{kedlaya-pottharst}.
The tricky part is to formulate the correct conclusion; we do not attempt to do this here.
\end{remark}

\subsection{Drinfeld's lemma}
\label{subsec:Drinfeld lemma}

The following arguments are similar to \cite[\S 7]{kedlaya-simpleconn}.

\begin{lemma} \label{L:multi-convexity}
Let $X$ be the subset of the analytic $m$-space over $K$ in the variables $T_1,\dots,T_m$ defined by the condition
$(\log |T_1|,\dots,\log |T_m|) \in U$ for some convex subset $U$ of $\RR^m$,
then set notation as in \S\ref{sec:relative}.
Let $\eta_{\rho_1,\dots,\rho_m} \in X$ be the Gauss point $|T_i| = e^{\rho_i}$.
Let $\calE$ be an $\calO_X$-linear connection of rank $n$ on $X \times_K A_K[I]$ for some $I$.
Define the functions $F_1,\dots,F_n\colon U \times I \to \RR$ by the formula 
\[
F_i(r_1,\dots,r_m,r) = -\log s_1(\calE_x, e^{-r}) - \cdots -\log s_i(\calE_x, e^{-r}),
\qquad 
x = \eta_{e^{-r_1},\dots,e^{-r_m}}.
\]
Then the functions $F_1,\dots,F_d$ are convex.
\end{lemma}
\begin{proof}
Immediate from Lemma~\ref{L:relative continuous}.
Alternatively, see \cite{kedlaya-xiao} for a slightly different approach to a similar result.
\end{proof}

\begin{defn}
Let $X$ be an adic space over $K$.
Choose a positive integer $m$ and a value $\epsilon \in (0,1)$.
By a \emph{partial relative Frobenius system}, we will mean a tuple $(\varphi_1,\dots,\varphi_m)$ where $\varphi_i\colon A_K(\epsilon,1)^m \to A_K(\epsilon,1)^m$ 
is a 
be the ``partial relative Frobenius'' map given by composing some isometric endomorphism of $K$ with a $K$-linear substitution of the form $t \mapsto t^q + u$
where $|u|_1 < 1$.
\end{defn}

\begin{lemma} \label{L:constant breaks}
Set $X := A_K(\epsilon,1)^m$ for some positive integer $m$ and some $\epsilon > 0$.
Let $\calE$ be an object of $\calC_X$ of rank $n$ such that for $i=1,\dots,m$,
there exists an isomorphism $\varphi_i^* \calE \cong \calE$ where $\varphi_i$ is the map induced by a relative Frobenius lift on the $i$-th factor of $X$.
\begin{enumerate}
\item[(a)]
The functions $b_i(\calE, \bullet)\colon X \to \RR$ are constant for $i=1,\dots,n$.
\item[(b)]
There is a direct sum decomposition of $\calE$ that specializes to the slope decomposition of $\calE_x$ for each $x \in X$.
\end{enumerate}
\end{lemma}
\begin{proof}
By choosing an affinoid subspace of $X$ whose Frobenius translates cover $X$,
we may find a single $\epsilon' \in (0,1)$ such that
$\calE$ can be realized as an $\calO_X$-linear connection on $X \times_K A_K(\epsilon',1)$.
By increasing $\epsilon'$, we can ensure that for some $x_0 \in X$,
we have $s_i(\calE_{x_0}, \rho) = \rho^{b_i(\calE_{x_0})}$ for $\rho \in (\epsilon', 1)$.

For each $\rho \in (\epsilon',1)$, Lemma~\ref{L:multi-convexity} implies that the functions
$F_1,\dots,F_n$ are convex on $(0, -\log \epsilon)^m \times (0, -\log \epsilon')$. 
On the other hand, the partial Frobenius structures implies that these functions are periodic on $(0, -\log \epsilon)^m$. Consequently, each $F_i$ depends only on the last argument.
By Lemma~\ref{L:relative continuous} again, we deduce further that
$-\log s_1(\calE_x, e^{-r}) - \cdots -\log s_i(\calE_x, e^{-r})$ is independent of $x \in X$
whether or not $x$ is a Gauss point.
Combining this with the previous paragraph, we deduce that $s_i(\calE_{x}, \rho) = \rho^{b_i(\calE_{x_0})}$ for $x \in X, \rho \in (\epsilon, 1)$.
This immediately yields (a); we deduce (b) by appealing to Lemma~\ref{L:global slope decomposition}.
\end{proof}

\begin{theorem} \label{T:drinfeld relative p-adic mono}
With notation as in Lemma~\ref{L:constant breaks},
for some $\epsilon' \in (0,1)$
there exists an eligible finite \'etale cover $Y$ of $A_K(\epsilon', 1)$ such that the pullback of $\calE$ along $X \times_K Y \to X \times_K A_K(\epsilon',1)$
is regular.
\end{theorem}
\begin{proof}
By applying Theorem~\ref{T:p-adic mono} at some single point $x_0$ of $X$, we obtain a cover that makes $\calE_{x_0}$ regular.
Then Lemma~\ref{L:constant breaks} implies that this cover makes every $\calE_x$ regular for all $x \in X$.
\end{proof}

\begin{cor} \label{C:drinfeld relative p-adic mono relative}
Let $X$ be an adic space over $K$ and put $X' := X \times_K A_K(\epsilon,1)^m$.
Let $\calE$ be an object of $\calC_{X'}$ of rank $n$ such that for $i=1,\dots,m$,
there exists an isomorphism $\varphi_i^* \calE \cong \calE$ where $\varphi_i$ is the map induced by a relative Frobenius lift on the $i$-th factor of $A_K(\epsilon,1)^m$.
Then for each $x_0 \in X$, there exist an open neighborhood $U$ of $x_0$ in $X$,
a value $\epsilon' \in (0,1)$,
and an eligible finite \'etale cover $Y$ of $U \times_K A_K(\epsilon',1)$ such that 
the pullback of $\calE$ from $X' \times_K A_K(\epsilon',1) \cong 
X \times_K A_K(\epsilon',1) \times A_K(\epsilon,1)^m$ to $Y \times_K A_K(\epsilon,1)^m$
is regular.
\end{cor}
\begin{proof}
This follows immediately by combining Theorem~\ref{T:drinfeld relative p-adic mono},
Lemma~\ref{L:spread eligible cover}, and Theorem~\ref{T:break semicontinuity}.
\end{proof}

Note that in the previous results, we have products of $m+1$ annuli with one factor being treated separately. If we symmetrize the hypothesis, we end up with a stronger conclusion.
For clarity, we first state a restricted version of the result and one of its corollaries, then immediately state and prove the full result.

\begin{theorem} \label{T:Drinfeld lemma form of local monodromy}
Let $\calE$ be a vector bundle on $A_K(\epsilon,1)^m$ for some $\epsilon \in (0,1)$ equipped with an integrable $K$-linear connection.
Suppose that for $i=1,\dots,m$, $\calE$ is isomorphic to its pullback along the map on $A_K(\epsilon,1)^m$ induced by a partial Frobenius lift on the $i$-th factor. 
Then for some $\epsilon' \in (\epsilon,1)$, there exists a finite eligible cover $Y$ of $A_K(\epsilon',1)$ such that the pullback of $\calE$ to $Y \times_K \cdots \times_K Y$ is unipotent (as a connection).
\end{theorem}
\begin{proof}
This is a special case of Theorem~\ref{T:Drinfeld lemma form of local monodromy relative} below.
\end{proof}
\begin{cor}
With notation as in Theorem~\ref{T:Drinfeld lemma form of local monodromy},
suppose further that $\calE$ is irreducible.
Then there exist quasiconstant objects $\calF_1,\dots,\calF_m \in \calC_K$ 
such that $\calE \cong \calF_1 \boxtimes \cdots \boxtimes \calF_m$.
\end{cor}
\begin{proof}
The pushforward of the trivial connection from $Y$ to $A_K(\epsilon',1)$ is semisimple,
so it decomposes as a direct sum $\bigoplus_i \calF_i$ of irreducible objects in $\calC_K$. By 
Theorem~\ref{T:Drinfeld lemma form of local monodromy} plus adjunction, $\calE$ admits a nonzero map
to the $m$-th box power of $\bigoplus_i \calF_i$; this proves the claim.
\end{proof}

\begin{theorem} \label{T:Drinfeld lemma form of local monodromy relative}
Let $X$ be an adic space over $K$.
Let $\calE$ be a vector bundle on $X \times_K A_K(\epsilon,1)^m$ for some $\epsilon \in (0,1)$ equipped with an integrable $\calO_X$-linear connection.
Suppose that for $i=1,\dots,m$, $\calE$ is isomorphic to its pullback along the map on $X \times_K A_K(\epsilon,1)^m$ induced by a partial Frobenius lift on the $i$-th factor of $A_K(\epsilon,1)^m$. 
Then for each $x_0 \in X$, there exist an open neighborhood $U$ of $x_0$, some $\epsilon' \in (\epsilon,1)$, and a finite eligible cover $Y$ of $U \times_K A_K(\epsilon',1)$ such that the pullback of $\calE$ from
\[
X \times_K A_K(\epsilon,1)^m \cong (X \times_K A_K(\epsilon,1)) \times_X \cdots \times_X (X \times_K A_K(\epsilon,1))
\]
to $Y \times_X \cdots \times_X Y$ is unipotent.
\end{theorem}
\begin{proof}
For each $i$, write $X \times_K A_K(\epsilon,1)^m$ as a product $X_i \times_K A_K(\epsilon,1)$ in which $X_i \cong X \times_K A_K(\epsilon,1)^{m-1}$ omits the $i$-th factor of $A_K(\epsilon,1)^m$.
By Lemma~\ref{L:Frobenius to solvable} (using the $i$-th partial Frobenius structure), $\calE \in \calC_{X_i}$
and so we may apply Corollary~\ref{C:drinfeld relative p-adic mono relative} (using the other partial Frobenius structures) to obtain
(after shrinking $X$ towards $x_0$) a cover $Y_i$ of $X \times_K A_K(\epsilon,1)$ that makes $\calE$ regular in the $i$-th direction.
By Theorem~\ref{T:p-adic mono with Frob} (using the $i$-th partial Frobenius structure), each fiber can be made unipotent by a further tame cover; by Theorem~\ref{T:unipotent spread}, the latter can be chosen uniformly. 
Now choose a single $Y$ that dominates each of the $Y_i$;
pulling back along $Y$ makes $\calE$ fiberwise unipotent in each direction, and we may argue as in the proof of Theorem~\ref{T:relative de Rham} to upgrade ``fiberwise unipotent'' to ``unipotent''.
\end{proof}

\begin{remark}
Theorem~\ref{T:Drinfeld lemma form of local monodromy} can be used
to establish a local analogue of Drinfeld's lemma for overconvergent $F$-isocrystals. See \cite{kedlaya-xu}.
\end{remark}

\begin{remark} \label{R:multivariate Cst}
Brinon, Chiarellotto, and Mazzari \cite{bcm} have introduced the analogue of the de Rham period ring for representations of the group $G_{K, \Delta}$, and established an analogue of Berger's construction of the associated differential equation. This produces a connection to which Theorem~\ref{T:Drinfeld lemma form of local monodromy} may be applied; it should be possible to deduce from this that any de Rham representation of $G_{K,\Delta}$ is potentially semistable.
\end{remark}


\begin{thebibliography}{99}
\bibitem{abe-companion}
T. Abe, Langlands correspondence for isocrystals and existence of
crystalline companion for curves, \textit{J. Amer. Math. Soc.} \textbf{31} (2018), 921--1057.

\bibitem{andre-mono}
Y. Andr\'e, Filtrations de type Hasse-Arf et monodromie $p$-adique,
\textit{Invent. Math.} \textbf{148} (2002), 285--317.

\bibitem{andreatta-brinon}
F. Andreatta and O. Brinon, Surconvergence des repr\'esentations $p$-adiques: le cas relatif,
\textit{Ast\'erisque} \textbf{319} (2008), 39--116.

\bibitem{bartenwerfer}
W. Bartenwerfer, $k$-holomorphe Vektorraumb\"undel auf offenen Polyzylindern,
\textit{J. reine angew Math.} \textbf{326} (1981), 214--220.

\bibitem{berger-mono}
L. Berger, Repr\'esentations $p$-adiques et \'equations diff\'erentielles,
\textit{Invent. Math.} \textbf{148} (2002), 219--284.

\bibitem{berkovich1}
V. Berkovich, \textit{Spectral Theory and Analytic Geometry over Non-Archimedean Fields},
Math. Surveys and Monographs 33, Amer. Math. Soc., 1990.

\bibitem{berkovich2}
V. Berkovich, \'Etale cohomology for non-Archimedean analytic spaces,
\textit{Publ. Math. IH\'ES} \textbf{78} (1993), 5--161.

\bibitem{berkovich-contract}
V.G. Berkovich, Smooth $p$-adic analytic spaces are locally contractible,
\textit{Invent. Math.} \textbf{137} (1999), 1--84.

\bibitem{berthelot-mem}
P. Berthelot, G\'eom\'etrie rigide et cohomologie des vari\'et\'es alg\'ebriques de caract\'eristique $p$, in
Introductions aux cohomologies $p$-adiques (Luminy, 1982), 
\textit{M\'em. Soc. Math. France (N.S.)} \textbf{23} (1986), 7--32.

\bibitem{bcm}
O. Brinon, B. Chiarellotto, and N. Mazzari, Multivariable de Rham representations, Sen theory and $p$-adic differential equations, \textit{Math. Res. Letters} \textbf{31} (2024), 25--90.

\bibitem{dejong}
A.J. de Jong, \'Etale fundamental groups of non-archimedean analytic spaces,
\textit{Compos. Math.} \textbf{97} (1995), 89--118.

\bibitem{dejong-alterations}
A.J. de Jong, Smoothness, semi-stability, and alterations,
\textit{Publ. Math. IH\'ES} \textbf{83} (1996), 51--93.

\bibitem{guo-yang}
H. Guo and Z. Yang, Pointwise criteria of $p$-adic local systems,
arXiv:2409.19742v1 (2024).

\bibitem{sousperfectoid}
D. Hansen and K.S. Kedlaya, Sheafiness criteria for Huber rings, preprint
available at \url{https://kskedlaya.org}.

\bibitem{huber-book}
R. Huber, \textit{\'Etale Cohomology of Rigid Analytic Varieties and Adic Spaces},
Aspects of Math. 30, Springer Fachmedien, Wiesbaden, 1996.

\bibitem{katz-gabber}
N.M. Katz, Local-to-global extensions of representations of fundamental groups,
\textit{Ann. Inst. Fourier} \textbf{36} (1986), 69--106.

\bibitem{kedlaya-locmono}
K.S. Kedlaya, A $p$-adic local monodromy theorem,
\textit{Annals of Math.} \textbf{160} (2004), 93--184.

\bibitem{kedlaya-etale}
K.S. Kedlaya,
More \'etale covers of affine spaces in positive characteristic,
\textit{J. Alg. Geom.} \textbf{14} (2005), 187--192.

\bibitem{kedlaya-semi1}
K.S. Kedlaya, Semistable reduction for overconvergent $F$-isocrystals, I: Unipotence and logarithmic extensions, \textit{Compos. Math.} 143 (2007), 1164--1212;
errata, \cite{shiho-log}.

\bibitem{kedlaya-semi4}
K.S. Kedlaya, Semistable reduction for overconvergent F-isocrystals, 
IV: Local semistable reduction at nonmonomial valuations, 
\textit{Compos. Math.} \textbf{147} (2011), 467--523;
erratum, \cite[Remark~4.6.7]{kedlaya-connections}.

\bibitem{kedlaya-connections}
K.S. Kedlaya,
Local and global structure of connections on nonarchimedean curves, 
\textit{Compos. Math.} \textbf{151} (2015), 1096--1156;
errata (with Atsushi Shiho), \textit{ibid.} 153 (2017), 2658--2665. 

\bibitem{kedlaya-reified}
K.S. Kedlaya, Reified valuations and adic spectra, \textit{Res. Number Theory} \textbf{1} (2015), 1--42. 

\bibitem{kedlaya-aws}
K.S. Kedlaya, Sheaves, stacks, and shtukas, in
\textit{Perfectoid Spaces: Lectures from the 2017 Arizona Winter School},
Math. Surveys and Monographs 242, American Mathematical Society, 2019.

\bibitem{kedlaya-simpleconn}
K.S. Kedlaya, Simple connectivity of Fargues-Fontaine curves, 
\textit{Annales Henri Lebesgue} \textbf{4} (2021), 1203--1233. 

\bibitem{kedlaya-book}
K.S. Kedlaya, \textit{$p$-adic Differential Equations}, 2nd edition, Cambridge University Press, Cambridge, 2022.

\bibitem{kedlaya-dl-isocrystals}
K.S. Kedlaya, Drinfeld's lemma for $F$-isocrystals, I, 
\textit{Intl. Math. Res. Notices} (2024), article ID rnae039. 

\bibitem{kedlaya-liu-families}
K.S. Kedlaya and R. Liu, On families of $(\varphi, \Gamma)$-modules, 
\textit{Algebra and Number Theory} \textbf{4} (2010), 943--967. 

\bibitem{kedlaya-liu1}
K.S. Kedlaya and R. Liu, Relative $p$-adic Hodge theory: Foundations,
\textit{Astérisque} \textbf{371} (2015), 239 pages. 

\bibitem{kedlaya-liu2}
K.S. Kedlaya and R. Liu, Relative $p$-adic Hodge theory, II: Imperfect period rings,
arXiv:1602.06899v3 (2019).

\bibitem{kedlaya-pottharst}
K.S. Kedlaya and J. Pottharst, On categories of $(\varphi, \Gamma)$-modules,
in \textit{Algebraic Geometry: Salt Lake City 2015, Part 2}, Proc. Symp. Pure Math. 97, Amer. Math. Soc., Providence, 2018, 281--304. 

\bibitem{kedlaya-xiao}
K.S. Kedlaya and L. Xiao,  Differential modules on p-adic polyannuli (with Liang Xiao), 
\textit{J. Inst. Math. Jussieu} \textbf{9} (2010), 155--201; erratum, \textit{ibid.} \textbf{9} (2010), 669--671. 

\bibitem{lawrence-li}
B. Lawrence and S. Li, A counterexample to an optimistic guess about \'etale local systems, \textit{C.R. Acad. Sci. Paris} \textbf{359} (2021), 923--924.

\bibitem{lestum}
B. Le Stum, \textit{Rigid Cohomology}, Cambridge Tracts in Math. 172, Cambridge University Press, 2007.

\bibitem{kedlaya-xu}
K.S. Kedlaya and D. Xu, Drinfeld's lemma for $F$-isocrystals, II: Tannakian approach, 
\textit{Compos. Math.} \textbf{160} (2024), 90--119.

\bibitem{liu-zhu}
R. Liu and X. Zhu, Rigidity and a Riemann--Hilbert correspondence for $p$-adic local systems,
\textit{Invent. Math.} \textbf{2017} (207), 291--343.

\bibitem{mebkhout-mono}
Z. Mebkhout, Analogue $p$-adique du th\'eor\`eme du Turrittin
et le th\'eor\`eme de la monodromie $p$-adique, \textit{Invent. Math.}
\textbf{148} (2002), 319--351.

\bibitem{shiho-log}
A. Shiho, On logarithmic extension of overconvergent isocrystals,
\textit{Math. Ann.} \textbf{348} (2010), 467--512.

\bibitem{shimizu}
K. Shimizu, A $p$-adic monodromy theorem for de Rham local systems, 
\textit{Compos. Math.} \textbf{158} (2022), 2157–-2205.

\bibitem{tsuzuki}
N. Tsuzuki, Slope filtration of quasi-unipotent overconvergent $F$-isocrystals,
\textit{Ann. Inst. Fourier} \textbf{48} (1998), 379--412.

\end{thebibliography}
\end{document}